\documentclass[12pt, reqno]{amsart}
\usepackage{amssymb,amsthm,amsmath}
\usepackage[numbers,sort&compress]{natbib}
\usepackage{amssymb,amsmath}
\usepackage{amsfonts}
\usepackage{mathrsfs}
\usepackage{latexsym}
\usepackage{amssymb}
\usepackage{amsthm}
\usepackage{indentfirst}
\date{June 26, 2017}

\let\oldsection\section
\renewcommand\section{\setcounter{equation}{0}\oldsection}

\newtheorem{corollary}{Corollary}[section]
\newtheorem{theorem}{Theorem}[section]
\newtheorem{lemma}{Lemma}[section]
\newtheorem{proposition}{Proposition}[section]
\newtheorem{definition}{Definition}[section]

\newtheorem{remark}{Remark}[section]

\allowdisplaybreaks
\begin{document}

\title[The primitive equations as the small aspect ratio limit]{The primitive equations as the small aspect ratio limit of the Navier-Stokes equations: rigorous justification of the hydrostatic approximation}

\author{Jinkai~Li}
\address[Jinkai~Li]{Department of Mathematics, The Chinese University of Hong Kong, Hong Kong, P.R.China}
\email{jklimath@gmail.com}

\author{Edriss~S.~Titi}
\address[Edriss~S.~Titi]{
Department of Mathematics, Texas A\&M University, 3368--TAMU, College Station, TX 77843-3368, USA. ALSO, Department of Computer Science and Applied Mathematics, Weizmann Institute of Science, Rehovot 76100, Israel.}
\email{titi@math.tamu.edu and edriss.titi@weizmann.ac.il}



\begin{abstract}
An important feature of the planetary oceanic
dynamics is that the aspect ratio (the ratio of the depth to horizontal width) is very small.
As a result, the hydrostatic
approximation (balance), derived by performing the formal small aspect ratio limit
to the Navier-Stokes equations, is considered as a fundamental
component in the primitive equations of the
large-scale ocean.
In this paper, we justify rigorously
the small aspect ratio limit of the Navier-Stokes equations to
the primitive equations. Specifically, we prove that the
Navier-Stokes equations, after being scaled appropriately by the small aspect ratio parameter of the physical domain, converge strongly to the primitive
equations, globally and uniformly in time, and the convergence rate is of the
same order as the aspect ratio parameter. This result
validates the hydrostatic
approximation for the large-scale
oceanic dynamics. Notably, only the weak convergence of this small aspect ratio
limit was rigorously justified before.
\end{abstract}

\maketitle


\allowdisplaybreaks
\section{Introduction}
\label{sec1}
In the context of the geophysical flow concerning the large-scale
oceanic dynamics, the ratio of the depth to the horizontal width is very small.
With the aid of this fact, by scaling the incompressible Navier-Stokes
equations with respect to the aspect ratio parameter
and taking the small aspect
ratio limit, one obtains formally the primitive equations for the large-scale
oceanic dynamics. The primitive equations are nothing but the Navier-Stokes
equations in which the vertical momentum equation is being replaced
by the hydrostatic
approximation (balance). Due to its high accuracy for the large-scale oceanic dynamics,
the hydrostatic approximation
forms a fundamental component in the primitive equations.

The rigorous mathematical justification of the small aspect ratio limit
from the Navier-Stokes equations to the primitive equations was studied
before by Az\'erad--Guill\'en \cite{AZGU}, in which the weak
convergence was established. Since only the weak convergence was obtained in
\cite{AZGU}, no convergence rate was provided. The aim of this paper
is to show the strong convergence from the Navier-Stokes equations to
the primitive equations, as the aspect ratio parameter
goes to zero. Moreover, it will be
shown, the strong convergence is actually global and uniform
in time, and that the
convergence rate is of the same order as the aspect ratio parameter.

Let's consider the anisotropic Navier-Stokes equations
$$
\partial_tu+(u\cdot\nabla)u-\mu\Delta_Hu-\nu\partial_z^2u+\nabla p=0,
$$
in the $\varepsilon$-dependent domain
$\Omega_\varepsilon:=M\times(-\varepsilon,\varepsilon)$, where $\varepsilon>0$ is a very small parameter, and
$M=(0,L_1)\times(0,L_2)$, for two positive constants $L_1$ and $L_2$
of order $O(1)$ with respect to $\varepsilon$.
Here the vector field $u=(v, w)$, with $v=(v_1, v_2)$, is the velocity,
and the scalar function
$p$ is the pressure. Similar to the
case considered in Az\'erad--Guill\'en
\cite{AZGU}, we suppose that the horizontal viscous coefficient
$\mu$ and the vertical viscous coefficient $\nu$ have different orders,
that is $\mu=O(1)$ and $\nu=O(\varepsilon^2)$. We suppose, for
simplicity, that $\mu=1$ and $\nu=\varepsilon^2$. Note that it is
necessary to consider the above anisotropic viscosities scaling
in the horizontal and vertical
directions, so that the
Navier-Stokes equations converge to the primitive equations,
as the aspect ratio $\varepsilon$ goes to zero.
In fact, for the case when $(\mu,\nu)=O(1)$, it has been
shown in Bresh--Lemoine--Simon \cite{BRLESI} that the stationary Navier-Stokes equations
converge to a linear system with only vertical dissipation.

We first transform the above anisotropic Navier-Stokes equations,
defined on the $\varepsilon$-dependent domain $\Omega_\varepsilon$,
to a scaled Navier-Stokes equations defined on a fixed domain. To this end, we introduce the new unknowns
\begin{eqnarray*}
  &&u_\varepsilon=(v_\varepsilon, w_\varepsilon),\quad
  v_\varepsilon(x,y,z,t)=v(x,y,\varepsilon z,t),\\
  &&w_\varepsilon(x,y,z,t)=\frac{1}{\varepsilon}w(x,y,\varepsilon z,t),\quad
  p_\varepsilon(x,y,z,t)=p(x,y,\varepsilon z,t),
\end{eqnarray*}
for any $(x,y,z)\in \Omega:=M\times(-1,1)$, and for any $t\in(0,\infty)$.
Then, $u_\varepsilon=(v_\varepsilon, w_\varepsilon)$ and $p_\varepsilon$ satisfy
the following scaled Navier-Stokes equations (SNS)
\begin{equation}\label{SNS}
(SNS)~~\left\{
\begin{array}{l}
  \partial_t v_\varepsilon+(u_\varepsilon\cdot\nabla)v_\varepsilon -\Delta v_\varepsilon+\nabla_Hp_\varepsilon=0,\\
  \nabla_H\cdot v_\varepsilon+\partial_zw_\varepsilon=0,\\
  \varepsilon^2(\partial_tw_\varepsilon+u_\varepsilon\cdot\nabla w_\varepsilon-\Delta w_\varepsilon)+\partial_zp_\varepsilon=0,
\end{array}
\right.
\end{equation}
defined in the fixed domain $\Omega$. In addition, we consider
the periodic boundary value problem to (SNS), and thus, complement it
with the boundary and initial conditions
\begin{eqnarray}
  &v_\varepsilon, w_\varepsilon \mbox{ and }p_\varepsilon\mbox{ are periodic in }x,y,z,\label{bc}\\
  &(v_\varepsilon, w_\varepsilon)|_{t=0}=(v_0, w_0), \label{ic}
\end{eqnarray}
where $(v_0, w_0)$ is given.
Furthermore, for simplicity, we suppose in addition that the following symmetry condition holds
\begin{equation}\label{sc}
  v_\varepsilon, w_\varepsilon\mbox{ and }p_\varepsilon \mbox{ are even, odd and odd with respect to }z,\mbox{ respectively}.
\end{equation}
Note that this symmetry condition is preserved by the dynamics of
(SNS), in other words,
it is automatically satisfied as long as it is
satisfied initially. For this reason, throughout this paper, we always
suppose, without any further mention, that the initial horizontal velocity $v_0$ satisfies
\begin{equation*}
v_0\mbox{ is periodic in }x,y,z,\mbox{ and is even in }z. \label{sc0}
\end{equation*}

Throughout this paper, we used $\nabla_H$ and $\Delta_H$ to denote
the horizontal gradient and horizontal Laplacian, respectively,
that is $\nabla_H=(\partial_x,\partial_y)$ and $\Delta_H=\partial_x^2+\partial_y^2$. For $1\leq q\leq\infty$, and positive integer $k$, we denote by $L^q(\Omega)$ and $H^k(\Omega)$, respectively,
the standard Lebessgue and Sobolev spaces equipped with the standard norms. We use $L^2_\sigma(\Omega)$
to denote the space consisting of all divergence-free functions in $L^2(\Omega)$. Note that since we
consider the periodic boundary problems, all the functions considered in this paper are supposed to be periodic in the
spatial variables. For simplicity, we use the notation $\|\cdot\|_q$ and $\|\cdot\|_{q,M}$ to denote the $L^q(\Omega)$ and
$L^q(M)$ norms, respectively. Also, we will use the same notation to
denote both a space itself and its finite product spaces.

Following the same arguments as those
for the standard Navier-Stokes equations,
see, e.g., Constantin--Foias \cite{CONFO} and Temam \cite{TEMAMNS},
one can prove that, for any initial data $u_0=(v_0, w_0)\in L^2(\Omega)$, with $\nabla\cdot u_0=0$,
there is a global weak solution
$u$ to the scaled Navier-Stokes equations (\ref{SNS}), subject to the boundary and initial conditions (\ref{bc})--(\ref{ic}),
and if moreover, the initial data $u_0\in H^1(\Omega)$, it has a unique local in time strong solution, where the weak
solutions are defined as:

\begin{definition}\label{def}
Given $u_0=(v_0,w_0)\in L^2(\Omega)$, with $\nabla\cdot u_0=0$. A space periodic function $u$ is called a Leray-Hopf weak solution to (SNS), subject to (\ref{bc})--(\ref{ic}), if

  (i) it has the regularity that
  $$
  u\in C_w([0,\infty); L^2_\sigma(\Omega))\cap L^2_{\text{loc}}([0,\infty), H^1(\Omega)),
  $$
  where the subscript $w$ means weakly continuous,

  (ii) satisfies the energy inequality
  $$
  \|v(t)\|_2^2+\varepsilon^2\|w(t)\|_2^2+2\int_0^t(\|\nabla v\|_2^2 +\varepsilon^2\|\nabla w\|_2^2)ds\leq\|v_0\|_2^2+\varepsilon^2\|w_0\|_2^2,
  $$
  for a.e. $t\in[0,\infty)$,

  (iii) and the following integral identity holds
  \begin{align*}
    \int_Q&[-(v\cdot\partial_t\varphi_H+\varepsilon^2 w\partial_t\varphi_3)+((u\cdot\nabla)v \cdot\varphi_H+\varepsilon^2 u\cdot\nabla w\varphi_3)+\nabla v:\nabla\varphi_H\\
    &+ \varepsilon^2\nabla w\cdot\nabla\varphi_3]dxdydzdt=\int_\Omega(v_0\cdot\varphi_H(\cdot,0) +\varepsilon^2 w_0\varphi_3(\cdot,0))dxdydz,
  \end{align*}
  for any spatially periodic function $\varphi=(\varphi_H,\varphi_3)$, with $\varphi_H=(\varphi_1,\varphi_2)$, such that $\nabla\cdot\varphi=0$ and $\varphi\in C_0^\infty(\overline\Omega\times[0,\infty))$, where $Q:=\Omega\times(0,\infty)$.
\end{definition}

Formally, by taking the limit $\varepsilon\rightarrow0$ in (SNS), one obtains the following primitive equations (PEs)
\begin{equation}\label{PE}
(PEs)~\left\{
\begin{array}{l}
  \partial_tv+(u\cdot\nabla)v-\Delta v+\nabla_Hp=0,\\
  \nabla_H\cdot v+\partial_zw=0,\\
  \partial_zp=0,
\end{array}
\right.
\end{equation}
where $u=(v,w)$.

The primitive equations form a corner stone in many global circulation
models (GCM) and are used as the fundamental models for the weather
prediction, see, e.g., Haltiner--Williams \cite{HAWI}, Lewandowski \cite{LEWAN}, Majda \cite{MAJDA},
Pedlosky \cite{PED}, Vallis \cite{VALLIS}, Washington--Parkinson \cite{WP}, and Zeng \cite{ZENG}.
During the last three
decades, since the works Lions--Temam--Wang \cite{LTW92A,LTW92B,LTW95} in the 1990s, the primitive equations have been the subject of very
intensive mathematical research.
The current state of art concerning the primitive equations is that
they have global weak solutions (but the general uniqueness is still unclear except
for some special cases \cite{BGMR03,TACHIM,KPRZ,LITITIUNIQ}),
see Lions--Temam--Wang \cite{LTW92A,LTW92B,LTW95}, and have a unique
global strong solution, see Cao--Titi \cite{CAOTITI07}, Kukavica--Ziane \cite{KUZI1,KUZI2}
and Kobelkov \cite{KOB}, and see Hieber--Kashiwabara \cite{HIEKAS} and
Hieber--Hussien--Kashiwabara \cite{HIEHUSKAS} for some generalizations
in the $L^p$ settings. Some recent developments concerning the global
strong solutions to the primitive equations towards the direction of partial dissipation cases are made by Cao--Titi \cite{CAOTITI12} and Cao--Li--Titi \cite{CAOLITITI1,CAOLITITI2,CAOLITITI3,CAOLITITI4,CAOLITITI5}. Notably,
the works \cite{CAOLITITI3,CAOLITITI4,CAOLITITI5} show
that the horizontal viscosity turns out
to be more crucial than the vertical one for the global well-posedness,
because the results there show that
the merely horizontal viscosity is sufficient to guarantee the global
well-posedness of strong solutions to the primitive equations,
see Li--Titi \cite{LITITITCM,LITITITCMMOISTURE} for some related results and also a recent survey paper by Li--Titi \cite{LITITISURVEY} for more information. However, the invicid primitive
equations may develop finite time singularities, see Cao et al.\,\cite{CINT} and Wong \cite{WONG}.

Despite the important fact that the hydrostatic approximation plays a
crucial role in
the primitive equations, to the best of our knowledge, the only
known mathematical justification of its derivation, via the small aspect
ratio limit, is done in \cite{AZGU}, where only the weak convergence
is proved, and no convergence rate can be deduced there. Historically,
the most possible reason that only the weak convergence can be
established
in \cite{AZGU} is that the global existence of strong solutions to
the primitive equations was still an open question at that time. As it will be
seen below in the proof of our results, the global well-posedness of
strong solutions to the primitive equations plays a fundamental
role in the strong convergence of the Navier-Stokes equations to the
primitive equations.
The aim of this paper is to rigorously justify the strong convergence from (SNS) to (PEs), subject to the same
boundary and initial conditions (\ref{bc})--(\ref{sc}).

Before stating our main results, it is necessary to clarify some statements on the initial data $u_0=(v_0, w_0)$.
Recall that the solutions considered in this paper satisfy the symmetry condition (\ref{sc}),
so does the initial datum $u_0=(v_0,w_0)$. Since $w_0$ is odd in $z$, one has $w_0|_{z=0}=0$. Thus,
it follows from the incompressibility condition that $w_0$ can be uniquely determined as
\begin{equation}\label{ne00}
w_0(x,y,z)=-\int_0^z\nabla_H\cdot v_0(x,y,z')dz',
\end{equation}
for any $(x,y)\in M$ and $z\in(-1,1)$.
Due to this fact, throughout this paper, concerning the initial
velocity $u_0$, we only need to specify the horizontal components $v_0$, while the vertical component
$w_0$ is uniquely determined in terms of $v_0$ through (\ref{ne00}).
For this reason, we use, in this paper, both the statements ``initial data $u_0$" and ``initial data $v_0$", with (\ref{ne00}) is assumed
for the latter case.

Now, we are ready to state our main results. In case that the initial data $v_0\in H^1(\Omega)$, one can not generally
expect that $w_0$, determined by (\ref{ne00}), belongs to $H^1(\Omega)$. Instead, one should consider
$u_0=(v_0,w_0)$ as
a function in $L^2(\Omega)$, and thus can obtain a global weak solution $(v_\varepsilon, w_\varepsilon)$ to (SNS), subject
to (\ref{bc})--(\ref{sc}). For this case, we have the following theorem concerning the strong convergence:

\begin{theorem}\label{thm0}
Given a periodic function $v_0\in H^1(\Omega)$, such that
$$
\nabla_H\cdot\left(\int_{-1}^1v_0(x,y,z)dz\right)=0,\quad\int_\Omega v_0(x,y,z) dxdydz=0.
$$
Let $(v_\varepsilon,
w_\varepsilon)$ and $(v,w)$, respectively, be an arbitrary Leray-Hopf weak solution to (SNS) and the unique global strong solution to (PEs), subject to (\ref{bc})--(\ref{sc}). Denote by
$$
(V_\varepsilon, W_\varepsilon)=(v_\varepsilon-v, w_\varepsilon-w).
$$

Then, we have the a priori estimate
\begin{align*}
  \sup_{0\leq t<\infty}\|(V_\varepsilon,\varepsilon W_\varepsilon)\|_2^2(t)  +\int_0^{\infty}\|\nabla( V_\varepsilon,\varepsilon W_\varepsilon) \|_2^2(t)dt
  \leq C\varepsilon^2(\|v_0\|_2^2+\varepsilon^2\|w_0\|_2^2+1)^2,
\end{align*}
for any $\varepsilon\in(0,\infty)$, where $C$ is a positive constant depending only on $\|v_0\|_{H^1}$, $L_1$, and $L_2$. As a consequence, we have the following strong convergences
\begin{eqnarray*}
  &(v_\varepsilon,\varepsilon w_\varepsilon)\rightarrow (v,0),\mbox{ in }L^\infty(0,\infty; L^2(\Omega)),\\
  &(\nabla v_\varepsilon, \varepsilon\nabla w_\varepsilon,w_\varepsilon)\rightarrow(\nabla v,0,w),\mbox{ in }L^2(0,\infty;L^2(\Omega)),
\end{eqnarray*}
and the convergence rate is of the order $O(\varepsilon)$.
\end{theorem}

If we moreover suppose that $v_0\in H^2(\Omega)$, then $u_0=(v_0,w_0)\in H^1(\Omega)$,
with $w_0$ given by (\ref{ne00}). Then, by the same arguments as for the standard Navier-Stokes equations, see, e.g., \cite{TEMAMNS,CONFO}, one can obtain the unique local (in time) strong solution $(v_\varepsilon, w_\varepsilon)$ to (SNS), subject to
(\ref{bc})--(\ref{sc}). For this case, we have the following theorem concerning the strong convergence, in which the convergence
is stronger than
that in Theorem \ref{thm0}:

\begin{theorem}
  \label{thm}
Given a periodic function $v_0\in H^2(\Omega)$, such that
$$
\nabla_H\cdot\left(\int_{-1}^1v_0(x,y,z)dz\right)=0,\quad\int_\Omega v_0(x,y,z) dxdydz=0.
$$
Let $(v_\varepsilon,
w_\varepsilon)$ and $(v,w)$, respectively, be the unique local (in time) strong solution to (SNS) and the unique global strong
solution to (PEs), subject to (\ref{bc})--(\ref{sc}). Denote
$$
(V_\varepsilon, W_\varepsilon)=(v_\varepsilon-v, w_\varepsilon-w).
$$

Then, there is a positive constant $\varepsilon_0$ depending only on the initial norm $\|v_0\|_{H^2}$, $L_1$ and $L_2$, such that, for any
$\varepsilon\in(0,\varepsilon_0)$, the strong solution $(v_\varepsilon, w_\varepsilon)$ of (SNS) exists globally in time, and the following estimate holds
$$
\sup_{0\leq t<\infty}\|(V_\varepsilon,\varepsilon W_\varepsilon)\|_{H^1}^2+
\int_0^\infty\|\nabla(V_\varepsilon,\varepsilon W_\varepsilon) \|_{H^1}^2dt\leq C\varepsilon^2,
$$
for a constant $C$ depending only on $\|v_0\|_{H^2}$, $L_1$ and $L_2$.
As a consequence, the following strong convergences hold
\begin{eqnarray*}
  &(v_\varepsilon,\varepsilon w_\varepsilon)\rightarrow (v,0),\mbox{ in }L^\infty(0,\infty; H^1(\Omega)),\\
  &(\nabla v_\varepsilon, \varepsilon\nabla w_\varepsilon,w_\varepsilon)\rightarrow(\nabla v,0,w),\mbox{ in }L^2(0,\infty;H^1(\Omega)),\\
  &w_\varepsilon\rightarrow w,\mbox{ in }L^\infty(0,\infty;L^2(\Omega)),
\end{eqnarray*}
and the convergence rate is of the order $O(\varepsilon)$.
\end{theorem}

\begin{remark}
(i) Theorems \ref{thm0} and \ref{thm} show that the strong convergence of solutions of (SNS) to
the corresponding ones of (PEs) is global and uniform in time, and the convergence rate is of the
same order to the aspect ratio parameter $\varepsilon$. Moreover, the smoother the initial data is,
the stronger the norms, in which the convergence takes place. This
validates mathematically the accuracy of the hydrostatic approximation.

(ii) The assumption $\int_\Omega v_0dxdydz=0$
is imposed only for the simplicity of the proof, and the same result
still holds for the general case. One can follow the proof presented
in this paper, and establish the relevant a priori
estimates on $(v_\varepsilon-\bar v_{0\Omega})$ and $(v-\bar v_{0\Omega})$, instead of on $v_\varepsilon$ and $v$ themselves,
where $\bar v_{0\Omega}=\int_\Omega v_0dxdydz$.

(iii) Generally, if $v_0\in H^k$, with $k\geq2$, then one can show that
  $$
  \sup_{0\leq t<\infty}\|(V_\varepsilon,\varepsilon W_\varepsilon)\|_{H^{k-1}}^2(t)+
  \int_0^\infty\|\nabla(V_\varepsilon,\varepsilon W_\varepsilon) \|_{H^{k-1}}^2(t)dt\leq C\varepsilon^2,
  $$
for some positive constant $C$ depending only on $\|v_0\|_{H^k}$, $L_1$ and $L_2$. This can
be done by carrying out higher energy estimates to the difference system (\ref{DIFF-1})--(\ref{DIFF-3}), below.

(iv) Observing the smoothing effects of the (SNS) and (PEs) to the unique strong solutions, one can also
show, in Theorem \ref{thm} (but not in Theorem \ref{thm0}), the strong convergence in stronger norms,
away from the initial time,
in particular, $(v_\varepsilon, w_\varepsilon)\rightarrow(v,w)$, in $C^k(\overline\Omega
\times(T,\infty))$, for any given positive time $T$ and nonnegative integer $k$.
\end{remark}

The proofs of Theorems \ref{thm0} and \ref{thm} consist of two main ingredients:
the \textit{a priori} estimates
on the global strong solution $(v,w)$ to
(PEs), and the \textit{a priori} estimates on the difference $U_\varepsilon=(V_\varepsilon, W_\varepsilon):
=(v_\varepsilon, w_\varepsilon)-(v, w)$.
Since the convergences stated in the theorems are global and uniform in time, the desired \textit{a priori} estimates mentioned above should be global and uniform
in time. To this end, as it
has been used in several works before, see, e.g., \cite{CAOTITI07,CAOTITI12, CAOLITITI1,CAOLITITI2,CAOLITITI3,CAOLITITI4,CAOLITITI5}, we use
anisotropic treatments for (PEs) to get the \textit{a priori} estimates: we successively do
the basic energy estimate, the $L^\infty(0,\infty; L^4(\Omega))$
estimate on $v$, the $L^\infty(0,\infty; L^2(\Omega))$ estimate
on $\partial_zv$, $\nabla v$ and $\Delta v$, respectively,
where the hydrostatic approximation plays an essential role for obtaining the
the $L^\infty(0,\infty; L^4(\Omega))$ estimate on $v$, and
the Ladyzhenskaya type inequality (see Lemma \ref{LADTYPE}, below) is
frequently used throughout the whole proof. For the case of Theorem \ref{thm0},
the \textit{a priori} estimates up to $L^\infty(0,\infty; L^2(\Omega))$ of $\nabla v$ are enough, while
for Theorem \ref{thm}, we need one order higher estimates, that is $L^\infty(0,\infty; L^2(\Omega))$ of $\Delta v$.

The treatments on the estimates of the difference function $U_\varepsilon$
are different in the proofs of
Theorem \ref{thm0} and Theorem \ref{thm}. For the case of Theorem \ref{thm0}, since $(v_\varepsilon, w_\varepsilon)$
is only a Leray-Hopf weak solution, one can not do the subtraction of (SNS) and (PEs) and preform the
energy estimates to the system of difference, as it is usually done
for the strong solutions.
Instead, one can only perform the energy estimates in the framework of
the weak solutions. To this end, we adopt the idea, which was introduced
in Serrin \cite{SERRIN} (see also Bardos et al.\,\cite{BLNNT} and the
reference therein) to prove the weak-strong uniqueness
of the Navier-Stokes equations; however, the difference in our case is that,
the role of ``strong solutions" is now played by the solutions of the
(PEs), while the role of ``weak solutions" is now played by those of
(SNS), or intuitively, we are somehow doing the weak-strong uniqueness
between two different systems.
Precisely, we will: (i) use
$(v,w)$ as the testing functions for (SNS);
(ii) test (PEs) by $v_\varepsilon$; (iii) perform the basic energy identity of (PEs); (iv)
use the energy inequality for (SNS). Noticing that we end up with (i)--(iv) only one inequality but three equalities,
by manipulating these four formulas in a suitable way, we get the desired a priori
estimates for $U_\varepsilon$. We remark that this argument
can be viewed as the translation, to the language of weak solutions, of the approach of
performing energy estimates (for the strong solution) to the system of $U_\varepsilon$, below.

For the case of Theorem \ref{thm}, since the solutions considered are strong ones, one can
get the desired global in time estimates on $U_\varepsilon$ by using standard energy approach to
the system governing $(V_\varepsilon,W_\varepsilon)$, which reads as
 \begin{eqnarray*}
  &\partial_tV_\varepsilon+(U_\varepsilon\cdot\nabla)V_\varepsilon -\Delta V_\varepsilon+\nabla_HP_\varepsilon+( u\cdot\nabla)V_\varepsilon+(U_\varepsilon\cdot\nabla) v=0,\\
  &\nabla_H\cdot V_\varepsilon+\partial_zW_\varepsilon=0,\\
  &\varepsilon^2(\partial_tW_\varepsilon+U_\varepsilon\cdot\nabla W_\varepsilon-\Delta W_\varepsilon+U_\varepsilon\cdot\nabla w+ u\cdot\nabla W_\varepsilon)+\partial_zP_\varepsilon\\
  &=-\varepsilon^2(\partial_tw+u\cdot\nabla w-\Delta  w).
  \end{eqnarray*}
However, we will have to introduce some new ideas described in the
following steps. First, since the initial value of $(V_\varepsilon, W_\varepsilon)$
vanishes,
and there is a small coefficient $\varepsilon^2$ in the front of
the ``external forcing" terms in right-hand side of the above system,
one can perform the energy
approach and take advantage of the smallness argument to
get the desired a priori estimate on $(V_\varepsilon,
W_\varepsilon)$, and as a result,
the strong solution $(v_\varepsilon, w_\varepsilon)$ can be extended to be a global
one, for small $\varepsilon$. Second, one has to observe that, when
performing the energy estimates to the horizontal momentum equations for $V_\varepsilon$, no more information
about $W_\varepsilon$ can be used other than that comes from the
incompressibility condition; in other words, the term
$W_\varepsilon\partial_zV_\varepsilon$, in the horizontal momentum
equations for $V_\varepsilon$, can be only dealt with by expressing $W_\varepsilon$ as
$$
W_\varepsilon(x,y,z,t)=-\int_{0}^z\nabla_H\cdot V_\varepsilon(x,y,z',t)dz'.
$$
This is because the explicit dynamical information of $W_\varepsilon$, which comes from the vertical momentum equation,
is always
tied up with the parameter $\varepsilon$ (see Propositions \ref{prop},
\ref{basicdiff} and \ref{firstdiff}, below), which will finally go to
zero, in other words, the vertical momentum equation for $W_\varepsilon$ provides no $\varepsilon$-independent dynamical information of $W_\varepsilon$.
After achieving the desired \textit{a priori} estimates, the strong convergences follow immediately.

Throughout this paper, we always suppose, as it is stated in Theorem \ref{thm0} and Theorem \ref{thm}, that the initial data $v_0$ satisfies
$$
\int_\Omega v_0(x,y,z)dxdydz=0.
$$
As a result, by integrating the horizontal momentum equations of the (SNS) and (PEs) over $\Omega$, respectively, one has
$$
\int_\Omega v_\varepsilon(x,y,z,t)dxdydz=\int_\Omega v(x,y,z,t)dxdydz=0.
$$
Noticing that
\begin{eqnarray*}
&&w_\varepsilon(x,y,z,t)=-\int _0^z\nabla_H\cdot v_\varepsilon dz',\quad w(x,y,z,t)=-\int _0^z\nabla_H\cdot vdz',
\end{eqnarray*}
for any $(x,y,z)\in\Omega$, we have
$$
\int_\Omega w_\varepsilon(x,y,z,t)dxdydz=\int_\Omega w(x,y,z,t)dxdydz=0.
$$
Therefore, all the velocities encountered in this paper are of average zero. We will always, without any further mentions, recall this fact before using the Poincar\'e inequality, in the rest of this paper.

The rest of this paper is arranged as follows: some preliminary lemmas are collected in the next section, section \ref{sec2}. In
section \ref{sec3}, we carry out the \textit{a priori} estimates on the strong
solutions to (PEs), while the proofs of Theorem \ref{thm0} and Theorem \ref{thm} are given in section \ref{sec4} and section \ref{sec5}, respectively.

Throughout this paper, if not specified, the $C$ denotes a general positive constant depending only on $L_1$ and $L_2$.

\section{Preliminaries}
\label{sec2}
In this section, we state some Ladyzhenskaya-type inequalities for some
kinds of three dimensional integrals, which will be frequently used in the rest of this paper.

\begin{lemma}[see \cite{CAOTITI03}]\label{LADTYPE}
The following inequalities hold true
\begin{align*}
&\int_M\left(\int_{-1}^1f(x,y,z)dz\right)\left(\int_{-1}^1g(x,y,z)h(x,y,z)dz\right)dxdy\\
\leq&C\|f\|_2^{1/2}\left(\|f\|_2^{1/2}+\|\nabla_Hf\|_{2}^{1/2}\right)\|g\|_2\|h\|_2^{1/2}\left(
\|h\|_2^{1/2}+\|\nabla_Hh\|_2^{1/2}\right),
\end{align*}
and
\begin{align*}
&\int_M\left(\int_{-1}^1f(x,y,z)dz\right)\left(\int_{-1}^1g(x,y,z)h(x,y,z)dz\right)dxdy\\
\leq&C\|f\|_2\|g\|_2^{1/2}\left(\|g\|_2^{1/2}+\|\nabla_Hg\|_{2}^{1/2}\right)\|h\|_2^{1/2}\left(
\|h\|_2^{1/2}+\|\nabla_Hh\|_2^{1/2}\right),
\end{align*}
for every $f,g,h$ such that the right-hand sides make sense and are finite, where $C$ is a positive constant depending only on $L_1$ and $L_2$.
\end{lemma}

As a corollary, we prove the following:

\begin{lemma}\label{ladlem}
Let $\varphi=(\varphi_1, \varphi_2, \varphi_3), \phi$ and $\psi$ be periodic functions with basic domain $\Omega$. Suppose that $\varphi\in H^1(\Omega)$,
with $\nabla\cdot \varphi=0$ in $\Omega$, $\int_\Omega \varphi dxdydz=0$,
and $\varphi_3|_{z=0}=0$,
$\nabla \phi\in H^1(\Omega)$ and $\psi\in L^2(\Omega)$. Denote by
$\varphi_H=(\varphi_1, \varphi_2)$ the horizontal components of the function $\varphi$.
Then, we have the following estimate
$$
  \left|\int_\Omega (\varphi\cdot\nabla \phi)\psi dxdydz\right|\leq C\|\nabla \varphi_H\|_2^{\frac12}\|\Delta \varphi_H\|_2^{\frac12}\|\nabla \phi\|_2^{\frac12}\|\Delta \phi\|_2^{\frac12}\|\psi\|_2,
$$
  where $C$ is a positive constant depending only on $L_1$ and $L_2$.
\end{lemma}

\begin{proof}
  Since $\varphi_3(x,y,0)=0$ and $\nabla\cdot \varphi=0$, one has
  $$
  \varphi_3(x,y,z)=\int_0^z\partial_z\varphi_3(x,y,z')dz'=- \int_0^z\nabla_H\cdot \varphi_h(x,y,z')dz',
  $$
  from which, by the H\"older inequality, we have
  \begin{eqnarray*}
    &&\|\varphi_3\|_2,\|\partial_z\varphi_3\|_2\leq\|\nabla_H \varphi_H\|_2,\quad \|\nabla_H\varphi_3\|_2,\|\nabla_H\partial_z \varphi_3\|_2\leq\|\Delta_H\varphi_H\|_2,
  \end{eqnarray*}
  and thus, recalling that $\int_\Omega \varphi dxdydz=0$, it follows from the Poincar\'e inequality that
  \begin{align}
  \|\varphi\|_2\leq&\|\varphi_H\|_2+\|\varphi_3\|_2\leq \|\varphi_H\|_2+\|\nabla_H\varphi_H\|_2\leq C\|\nabla \varphi_H\|_2,\label{c1}\\
  \|\nabla_H\varphi\|_2\leq&\|\nabla_H\varphi_H \|_2+\|\nabla_H\varphi_3\|_2\nonumber\\
  \leq&\|\nabla_H\varphi_H\|_2+\|\Delta_H\varphi_H\|_2\leq C\|\Delta \varphi_H\|_2,\label{c2}\\
  \|\partial_z\varphi\|_2\leq&\|\partial_z\varphi_H \|_2+\|\partial_z\varphi_3\|_2\leq \|\partial_z\varphi_H\|_2+\|\nabla_H \varphi_H \|_2\leq C\|\nabla \varphi_H\|_2,\label{c3}\\
  \|\partial_z\nabla_H \varphi\|_2\leq&\|\partial_z\nabla_H\varphi_H \|_2+\|\partial_z\nabla_H\varphi_3\|_2\nonumber\\
  \leq& \|\partial_z\nabla_H\varphi_H\|_2+\|\Delta_H \varphi_H\|_2\leq C\|\Delta \varphi_H\|_2. \nonumber
  \end{align}
  Therefore, it follows from Lemma \ref{LADTYPE} and the Poincar\'e inequality that
  \begin{align*}
    &\left|\int_\Omega (\varphi\cdot\nabla \phi)\psi dxdydz\right|\\
    \leq&\int_M\left(\int_{-1}^1(|\varphi|+|\partial_z\varphi|) dz\right)\left(\int_{-1}^1|\nabla \phi||\psi|dz\right)dxdy\\
    \leq&C\left[\|\varphi\|_2^{\frac12}(\|\varphi\|_2+\|\nabla_H\varphi \|_2)^{\frac12}+ \|\partial_z\varphi\|_2^{\frac12}(\|\partial_z\varphi \|_2+\|\nabla_H\partial_z\varphi\|_2)^{\frac12}\right]\\
    &\times\|\nabla \phi\|_2^{\frac12}(\|\nabla \phi\|_2+\|\nabla_H\nabla \phi\|_2)^{\frac12}\|\psi\|_2\\
    \leq&C\|\nabla \varphi_H\|_2^{\frac12}\|\Delta \varphi_H\|_2^{\frac12}\|\nabla \phi\|_2^{\frac12}\|\Delta \phi\|_2^{\frac12}\|\psi\|_2,
  \end{align*}
  proving the conclusion.
\end{proof}

\section{A priori estimates on the primitive equations}
\label{sec3}
As it was mentioned in the introduction, the global well-posedness
(more precisely, the \textit{a priori} estimates) of strong solutions
to (PEs) plays a fundamental role in the proof of the strong convergences of the small aspect ratio limit of the Navier-Stoke equations to the primitive equations. In this section, we carry out the \textit{a priori} estimates on the strong solutions to the primitive equations.

We rewrite the primitive equations (\ref{PE}) as
\begin{eqnarray}
  &&\partial_tv+(v\cdot\nabla_H)v+w\partial_zv-\Delta v+\nabla_Hp(x,y,t)=0,\label{PE-1}\\
  &&\nabla_H\cdot v+\partial_zw=0.\label{PE-2}
\end{eqnarray}
Note that, we have used here the fact that, due to the identity $\partial_zp=0$ in (\ref{PE}),
the pressure depends only on two spatial variables $x$ and $y$.

By the $H^1$ theory of the primitive equations, see \cite{CAOTITI07}, for any $H^1$ initial data $v_0$, such that
$$
\nabla_H\cdot\left(\int_{-1}^1 v_0(x,y,z)dz\right)=0,\quad\mbox{for all }(x,y)\in M,
$$
there is a unique global strong solution $v$ to the primitive equations (\ref{PE-1})--(\ref{PE-2}), subject to (\ref{bc})--(\ref{sc}), such that $v\in C([0,\infty); H^1(\Omega))\cap L^2_{\text{loc}}([0,\infty); H^2(\Omega))$ and
$\partial_tv\in L^2_{\text{loc}}([0,\infty); L^2(\Omega))$. Generally, if the initial data $v_0$ has more regularities, then the solution $v$ will have the corresponding higher regularities, see, e.g., Petcu--Wirosoetisno \cite{PEWI05}. Moreover, due to the smoothing effect of the primitive equations to the strong solutions, one can show that $v$ is smooth away from the initial time, see Corollary 3.1 in Li--Titi \cite{LITITIUNIQ}. This fact guarantees the validity of our arguments in the following proofs.

We are going to do several a priori estimates on $v$, the unique global strong solution to the primitive equations (\ref{PE-1})--(\ref{PE-2}), subject to the boundary and initial conditions (\ref{bc})--(\ref{sc}), with initial data $v_0$.

Let's start with the following basic energy estimate.

\begin{proposition}[Basic energy estimate]
  \label{pebasic}
  Suppose that $v_0\in H^1(\Omega)$. Then, we have
  $$
  \|v(t)\|_2^2+2\int_0^t\|\nabla v\|_2^2ds=\|v_0\|_2^2,\quad\mbox{and }\quad
  \|v(t)\|_2^2\leq e^{-2\lambda_1t}\|v_0\|_2^2,
  $$
  for any $t\in[0,\infty)$, where $\lambda_1>0$ is the first eigenvalue of the following eigenvalue problem
  \begin{equation*}
    -\Delta\phi=\lambda\phi,\quad\int_\Omega\phi\, dxdydz=0,\quad\phi\mbox{ is periodic.}
  \end{equation*}
\end{proposition}

\begin{proof}
Taking the $L^2(\Omega)$ inner product to equation (\ref{PE-1})
with $v$, then it follows from integration by parts that
$$
  \frac12\frac{d}{dt}\|v\|_2^2+\|\nabla v\|_2^2=0,
$$
from which, integrating in $t$ yields the first conclusion. By
the Poincar\'e inequality, one has $\|\nabla v\|_2^2\geq\lambda_1\|v\|_2^2$, and thus we
have
  $$
  \frac{d}{dt}\|v\|_2^2+2\lambda_1\|v\|_2^2\leq0,
  $$
from which, by the Gronwall inequality, the second conclusion
follows.
\end{proof}

Since the high order estimates depend on the $L^\infty(0,\infty; L^4(\Omega))$ estimate of $v$, we first prove this estimate in the following lemma.

\begin{proposition}[$L^\infty(0,\infty; L^4(\Omega))$ estimate for $v$]
  \label{pel4}
  Suppose that $v_0\in H^1(\Omega)$. Then, we have the following estimate
  $$
  \sup_{0\leq s\leq t}\|v\|_4^4(s)+2\int_0^t\left\||v|\nabla v|\right\|_2^2(s)ds\leq e^{C(\|v_0\|_2^2+\|v_0\|_2^4)}\|v_0\|_4^4,
  $$
  for any $t\in[0,\infty)$, where $C$ is a positive constant depending only on $L_1$ and $L_2$.
\end{proposition}

\begin{proof}
  Multiplying equation (\ref{PE-1}) by $|v|^2v$, integrating the resultant over $\Omega$, then it follows from integration by parts that
  \begin{align}
  \frac14\frac{d}{dt}\|v\|_4^4&+\int_\Omega|v|^2(|\nabla v|^2+2|\nabla|v||^2)dxdydz\nonumber\\
  =&-\int_\Omega|v|^2v\cdot\nabla_Hp(x,y,t)dxdydz.\label{pe1}
  \end{align}
  By Lemma \ref{LADTYPE}, and using the Poincar\'e inequality, we have
  \begin{align}
    &-\int_\Omega|v|^2v\cdot\nabla_Hp(x,y,t)dxdydz\nonumber\\
    \leq&\int_M \left(\int_{-1}^1|v|^3dz\right)|\nabla_Hp(x,y,t)|dxdy\nonumber\\
    \leq&C\|\nabla_Hp\|_{2,M}\left\||v|^2\right\|_2^{\frac12}\left( \left\||v|^2\right\|_2+\left\|\nabla_H|v|^2\right\|_2\right)^{\frac12} \|v\|_2^{\frac12}(\|v\|_2+\|\nabla v\|_2)^{\frac12}\nonumber\\
    \leq&C\|\nabla_Hp\|_{2,M}\left\|v\right\|_4\left( \left\|v\right\|_4^2+\left\|\nabla_H|v|^2\right\|_2\right)^{\frac12} \|v\|_2^{\frac12}\|\nabla v\|_2^{\frac12}.\label{pe1-1}
  \end{align}
  Applying the operator $\int_{-1}^1\text{div}_H(\cdot)\,dz$ to equation (\ref{PE-1}), one obtains
  $$
  -\Delta_Hp(x,y,t)=\int_{-1}^1\nabla_H\cdot\Big(\nabla_H\cdot(v(x,y,z,t)\otimes v(x,y,z,t))\Big)dz.
  $$
  Note that $p$ can be uniquely determined by requiring $\int_\Omega p dxdy=0$, and thus,
  by the elliptic estimates, we have
  $$
  \|\nabla_Hp\|_{2,M}\leq C\left\|\int_{-1}^1\nabla_H\cdot\Big(\nabla_H\cdot(v\otimes v)\Big)dz\right\|_{2,M}\leq C\left\||v|\nabla_Hv\right\|_2.
  $$
  Thanks to the above estimate, it follows from (\ref{pe1-1}) and the Young inequality that
  \begin{align*}
  &-\int_\Omega|v|^2v\cdot\nabla_Hp(x,y,t)dxdydz\nonumber\\
    \leq&C\left\||v|\nabla_Hv\right\|_2\left\|v\right\|_4\left( \left\|v\right\|_4^2+\left\|\nabla_H|v|^2\right\|_2\right)^{\frac12} \|v\|_2^{\frac12}\|\nabla v\|_2^{\frac12}\\
    \leq&C(\|v\|_4^2\left\||v|\nabla_Hv\right\|_2 +\|v\|_4\left\||v|\nabla_Hv\right\|_2^{\frac32})\|v\|_2^{\frac12}\|\nabla v\|_2^{\frac12}\\
    \leq&\frac12\left\||v|\nabla v\right\|_2^2+C(\|v\|_2\|\nabla v\|_2+\|v\|_2^2\|\nabla v\|_2^2)\|v\|_4^4,
    \end{align*}
    which, substituted into (\ref{pe1}), gives
    \begin{align*}
      &\frac{d}{dt}\|v\|_4^4+2\left\||v|\nabla v\right\|_2^2\leq  C(\|v\|_2\|\nabla v\|_2+\|v\|_2^2\|\nabla v\|_2^2)\|v\|_4^4.
    \end{align*}
    Applying the Gronwall inequality to the above inequality, it follows from the H\"older inequality and Proposition \ref{pebasic} that
    \begin{align*}
      \sup_{0\leq s\leq t}\|v\|_4^4(s)+2\int_0^t\left\||v|\nabla v\right\|_2^2ds
      \leq&e^{C\int_0^t(\|v\|_2\|\nabla v\|_2+\|v\|_2^2\|\nabla v\|_2^2)ds}\|v_0\|_4^4\\
      \leq&e^{C\left[\left(\int_0^t\|v\|_2^2ds\right)^{1/2} \left(\int_0^t\|\nabla v\|_2^2ds\right)^{1/2}+\int_0^t \|v\|_2^2\|\nabla v\|_2^2)ds\right]}\|v_0\|_4^4\\
      \leq&\exp\{C(t^{1/2}e^{-\lambda_1t}\|v_0\|_2^2+\|v_0\|_2^4)\}\|v_0\|_4^4\\
      \leq&\exp\{C(\|v_0\|_2^2+\|v_0\|_2^4)\}\|v_0\|_4^4,
    \end{align*}
    proving the conclusion.
\end{proof}

Next, we work on the $L^\infty(0,\infty; L^2(\Omega))$ estimate for $\partial_zv$.

\begin{proposition}[$L^\infty(0,\infty; L^2(\Omega))$ estimate on $\partial_zv$]
  \label{pezv}
  Suppose that $v_0\in H^1(\Omega)$. Then, we have the following estimate
  \begin{align*}
    \sup_{0\leq s\leq t}\|\partial_zv\|_2^2(s)+\int_0^t\|\nabla\partial_zv\|_2^2ds
     \leq\|\partial_zv_0\|_2^2+C\|v_0\|_2^2\|v_0\|_4^8 e^{C(\|v_0\|_2^2+\|v_0\|_2^4)} ,
  \end{align*}
  for any $t\in[0,\infty)$, where $C$ is a positive constant depending only on $L_1$ and $L_2$.
\end{proposition}

\begin{proof}
  Taking the $L^2(\Omega)$ inner product of equation (\ref{PE-1}) with $-\partial_z^2v$, it follows from integration by parts that
  \begin{align*}
    \frac12\frac{d}{dt}\|\partial_zv\|_2^2+\|\nabla\partial_zv\|_2^2 =&-\int_\Omega[(\partial_zv\cdot\nabla_H) v-\nabla_H\cdot v\partial_zv]\cdot\partial_zvdxdydz\\
    \leq&4\int_\Omega|\partial_zv||\nabla_H\partial_zv||v|dxdydz.
  \end{align*}
  By the H\"older, Sobolev, Poincar\'e and Young inequalities, we have
  \begin{align*}
    &4\int_\Omega|\partial_zv||\nabla_H\partial_zv||v|dxdydz\leq 4\|\partial_zv\|_4\|v\|_4\|\nabla_H\partial_zv\|_2\\
    \leq& C\|v\|_4\|\partial_z v\|_2^{\frac14}\|\nabla\partial_zv\|_2^{\frac74}\leq \frac12\|\nabla\partial_zv\|_2^2+C\|v\|_4^8\|\partial_zv\|_2^2.
  \end{align*}
  Therefore, one obtains
  $$
  \frac{d}{dt}\|\partial_zv\|_2^2+\|\nabla\partial_zv\|_2^2\leq C\|v\|_4^8\|\partial_zv\|_2^2,
  $$
  from which, integrating in $t$ and using Propositions \ref{pebasic}--\ref{pel4}, we have
  \begin{align*}
    \sup_{0\leq s\leq t}\|\partial_zv\|_2^2(s)+\int_0^t\|\nabla\partial_zv\|_2^2ds \leq&\|\partial_zv_0\|_2^2+C\sup_{0\leq s\leq t}\|v\|_4^8\int_0^t\|\partial_zv\|_2^2ds\\
    \leq&\|\partial_zv_0\|_2^2+C\|v_0\|_2^2\|v_0\|_4^8 e^{C(\|v_0\|_2^2+\|v_0\|_2^4)},
  \end{align*}
  proving the conclusion.
\end{proof}

Then, we can establish the $L^\infty(0,\infty; H^1(\Omega))$ estimate on $v$.

\begin{proposition}[First order energy estimate]
  \label{pefirst}
  Suppose that $v_0\in H^1(\Omega)$. Then, we have the following estimate
  \begin{align*}
    &\sup_{0\leq s\leq t}\|\nabla v\|_2^2(s)+\frac12\int_0^t(\|\Delta v\|_2^2+\|\partial_tv\|_2^2)ds\\
    \leq& \|\nabla v_0\|_2^2 \exp\left\{C\left(\|v_0\|_2^4+\|\partial_zv_0\|_2^4+\|v_0\|_2^4\|v_0\|_4^{16} e^{C(\|v_0\|_2^2+\|v_0\|_2^4)}\right)\right\},
  \end{align*}
  for any $t\in[0,\infty)$, where $C$ is a positive constant depending only on $L_1$ and $L_2$.
\end{proposition}

\begin{proof}
  Taking the $L^2(\Omega)$ inner product to equation (\ref{PE-1}) with $\partial_tv-\Delta v$, then it follows from integration by parts that
  \begin{align}
    &\frac{d}{dt}\|\nabla v\|_2^2+\|\Delta v\|_2^2+\|\partial_tv\|_2^2=\int_\Omega[(v\cdot\nabla_H)v+w\partial_zv]\cdot (\Delta v-\partial_tv)dxdydz. \label{pe2}
  \end{align}
  By Lemma \ref{LADTYPE}, it follows from the Poincar\'e and Young inequalities that
  \begin{align}
    &\int_\Omega(v\cdot\nabla_H)v\cdot(\Delta v-\partial_tv)dxdydz\nonumber\\
    \leq&\int_M\left(\int_{-1}^1(|v|+|\partial_zv|)dz\right)\left(\int_{-1}^1|\nabla_H v|(|\Delta v|+|\partial_tv|)dz\right)dxdy\nonumber\\
    \leq&C\left[\|v\|_2^{\frac12}(\|v\|_2+\|\nabla_Hv\|_2)^{\frac12}+\|\partial_z v\|_2^{\frac12}(\|\partial_zv\|_2+\|\nabla_H\partial_zv\|_2)^{\frac12}\right]
    \nonumber\\
    &\times\|\nabla_Hv\|_2^{\frac12}(\|\nabla_Hv\|_2+\|\nabla_H^2v\|_2)^{\frac12} (\|\Delta v\|_2+\|\partial_tv\|_2)\nonumber\\
    \leq&C(\|v\|_2^{\frac12}\|\nabla v\|_2^{\frac12}+\|\partial_zv\|_2^{\frac12}\|\nabla\partial_zv\|_2^{\frac12}) \|\nabla_Hv\|_2\|\Delta v\|_2^{\frac12}(\|\Delta v\|_2+\|\partial_tv\|_2)\nonumber\\
    \leq&\frac14(\|\Delta v\|_2^2+\|\partial_tv\|_2^2)+C(\|v\|_2^2\|\nabla v\|_2^2+\|\partial_z v\|_2^2\|\nabla\partial_zv\|_2^2)\|\nabla v\|_2^2.\label{pe2-1}
  \end{align}
  Since $w$ is odd in $z$, it has $w|_{z=0}=0$, and thus
  $$
  w(x,y,z,t)=\int_0^z\partial_zw(x,y,z',t)dz'=-\int_0^z\nabla_H\cdot v(x,y,z',t)dz'.
  $$
  Thanks to this, it follows from Lemma \ref{LADTYPE}, the Poincar\'e and Young inequalities that
  \begin{align}
    &\int_\Omega w\partial_zv\cdot(\Delta v-\partial_tv)dxdydz\nonumber\\
    \leq&\int_M\left(\int_{-1}^1|\nabla_H\cdot v|dz\right)\left(\int_{-1}^1|\partial_zv| (|\Delta v|+|\partial_tv|)dz\right)dxdy\nonumber\\
    \leq&C\|\nabla_Hv\|_2^{\frac12}(\|\nabla_Hv\|_2+\|\nabla_H^2v\|_2)^{\frac12} \|\partial_zv\|_2^{\frac12}\nonumber\\
    &\times(\|\partial_zv\|_2+\|\nabla_H\partial_zv\|_2) ^{\frac12}(\|\Delta v\|_2+\|\partial_tv\|_2)\nonumber\\
    \leq&C\|\nabla_Hv\|_2^{\frac12}\|\Delta_Hv\|_2^{\frac12}\|\partial_zv\|_2 ^{\frac12}\|\nabla\partial_zv\|_2^{\frac12}(\|\Delta v\|_2+\|\partial_tv\|_2)\nonumber\\
    \leq&\frac14(\|\Delta v\|_2^2+\|\partial_tv\|_2^2)+C\|\partial_zv\|_2^2\|\nabla\partial_zv\|_2^2\|\nabla v\|_2^2.\label{pe2-2}
  \end{align}
  Substituting (\ref{pe2-1})--(\ref{pe2-2}) into (\ref{pe2}) yields
  \begin{align*}
    \frac{d}{dt}\|\nabla v\|_2^2+\frac12(\|\Delta v\|_2^2+\|\partial_tv\|_2^2)\leq C(\|v\|_2^2\|\nabla v\|_2^2+\|\partial_zv\|_2^2\|\nabla\partial_zv\|_2^2)\|\nabla v\|_2^2,
  \end{align*}
  from which, by the Gronwall inequality, it follows from Propositions \ref{pebasic} and \ref{pezv} that
  \begin{align*}
    &\sup_{0\leq s\leq t}\|\nabla v\|_2^2(s)+\frac12\int_0^t(\|\Delta v\|_2^2+\|\partial_tv\|_2^2)ds\\
    \leq&\exp\left\{C\int_0^t(\|v\|_2^2\|\nabla v\|_2^2+\|\partial_zv\|_2^2\|\nabla\partial_zv\|_2^2)ds\right\}\|\nabla v_0\|_2^2\\
    \leq&\|\nabla v_0\|_2^2 \exp\left\{C\left(\|v_0\|_2^4+\|\partial_zv_0\|_2^4+\|v_0\|_2^4\|v_0\|_4^{16} e^{C(\|v_0\|_2^2+\|v_0\|_2^4)}\right)\right\},
  \end{align*}
  proving the conclusion.
\end{proof}

And finally, in case that $v_0\in H^2(\Omega)$, we can obtain the second order energy estimate on $v$.

\begin{proposition}[Second order energy estimate]
  \label{pesecond}
  Suppose that $v_0\in H^2(\Omega)$. Then, we have
  \begin{align*}
    &\sup_{0\leq s\leq t}\|\Delta v\|_2^2(s)+\frac12\int_0^t(\|\nabla\Delta v\|_2^2+\|\nabla\partial_tv\|_2^2)ds\\
    \leq&\exp\left\{C\|\nabla v_0\|_2^4e^{C\left(\|v_0\|_2^4+\|\partial_z v_0\|_2^4+\|v_0\|_2^4\|v_0\|_4^{16}e^{C(\|v_0\|_2^2+\|v_0\|_2^4)}\right) }\right\}\|\Delta v_0\|_2^2,
  \end{align*}
  for any $t\in[0,\infty)$, where $C$ is a positive constant depending only on $L_1$ and $L_2$.
\end{proposition}

\begin{proof}
  Taking the $L^2(\Omega)$ inner product to equation (\ref{PE-1}) with $\Delta(\Delta v-\partial_tv)$ and integration by parts, then it follows from Lemma \ref{ladlem} and the Young inequality that
  \begin{align*}
    &\frac{d}{dt}\|\Delta v\|_2^2+\|\nabla\Delta v\|_2^2+\|\nabla\partial_t v\|_2^2\\
    =&\int_\Omega\nabla[(u\cdot\nabla)v]:\nabla(\Delta v-\partial_tv)dxdydz\\
    =&\int_\Omega[(\partial_iu\cdot\nabla)v+(u\cdot\partial_i\nabla) v]\cdot \partial_i(\Delta v-\partial_tv)dxdydz\\
    \leq&C(\|\partial_i\nabla v\|_2^{\frac12}\|\partial_i\Delta v\|_2^{\frac12}\|\nabla v\|_2^{\frac12}\|\Delta v\|_2^{\frac12}+\|\nabla v\|_2^{\frac12}\|\Delta v\|_2^{\frac12}\\
    &\times\|\partial_i\nabla v\|_2^{\frac12}\|\partial_i\Delta v\|_2^{\frac12})(\|\partial_i\Delta v\|_2+\|\partial_i\partial_tv\|_2)\\
    \leq&\frac{1}{2}(\|\nabla\Delta v\|_2^2+\|\nabla\partial_tv\|_2^2)+C\|\nabla v\|_2^2\|\Delta v\|_2^4,
  \end{align*}
  and thus
  $$
  \frac{d}{dt}\|\Delta v\|_2^2+\frac12(\|\nabla\Delta v\|_2^2+\|\nabla\partial_t v\|_2^2)\leq C\|\nabla v\|_2^2\|\Delta v\|_2^4.
  $$
  Applying the Gronwall inequality to the above inequality, it follows from Proposition \ref{pefirst} that
  \begin{align*}
    &\sup_{0\leq s\leq t}\|\Delta v\|_2^2(s)+\frac12\int_0^t(\|\nabla\Delta v\|_2^2+\|\nabla\partial_tv\|_2^2)ds\\
    \leq&\exp\left\{C\int_0^t\|\nabla v\|_2^2\|\Delta v\|_2^2ds\right\}\|\Delta v_0\|_2^2\\
    \leq&\exp\left\{C\|\nabla v_0\|_2^4e^{C\left(\|v_0\|_2^4+\|\partial_z v_0\|_2^4+\|v_0\|_2^4\|v_0\|_4^{16}e^{C(\|v_0\|_2^2+\|v_0\|_2^4)}\right) }\right\}\|\Delta v_0\|_2^2,
  \end{align*}
  proving the conclusion.
\end{proof}

As a direct corollary of Propositions \ref{pebasic}--\ref{pesecond}, we have the following:

\begin{corollary}
  \label{cor}
Suppose that $v_0\in H^m$, with $m=1$ or $m=2$. Let $(v,w)$ be the unique global strong solution to the primitive equations (\ref{PE-1})--(\ref{PE-2}), subject to (\ref{bc})--(\ref{sc}). Then, we have the following:

(i) If $v_0\in H^1(\Omega)$, then we have the estimate
$$
\sup_{0\leq t<\infty}\|v\|_{H^1}^2(t)+\int_0^\infty(\|\nabla v\|_{H^1}^2+
\|\partial_tv\|_2^2)dt\leq C(\|v_0\|_{H^1}, L_1, L_2);
$$
(ii) If $v_0\in H^2(\Omega)$, then we have the estimate
$$
\sup_{0\leq t<\infty}\|v\|_{H^2}^2(t)+\int_0^\infty(\|\nabla v\|_{H^2}^2+
\|\partial_tv\|_{H^1}^2)dt\leq C(\|v_0\|_{H^2}, L_1, L_2).
$$
\end{corollary}

\section{Strong convergence I: the $H^1$ initial data case}
\label{sec4}
This section is devoted to the strong convergence of (SNS) to (PEs), with initial data $v_0\in H^1(\Omega)$, in other words, we give the proof of Theorem \ref{thm0}.

Let the initial data $v_0\in H^1(\Omega)$, and assume
\begin{equation}\label{E1}
\nabla_H\cdot\left(\int_{-1}^1 v_0(x,y,z)dz\right)=0,\quad\mbox{for all }(x,y)\in M,
\end{equation}
by the $H^1$ theory of the primitive equations, see \cite{CAOTITI07}, there is a unique global strong solution $(v,w)$ to (PEs), subject to the boundary and initial conditions (\ref{bc})--(\ref{sc}), such that
\begin{eqnarray}
&v\in C([0,\infty); H^1(\Omega))\cap L^2_{\text{loc}}([0,\infty); H^2(\Omega)),\quad\partial_tv\in L^2_{\text{loc}}([0,\infty); L^2(\Omega)).\label{ne0}
\end{eqnarray}
Then, using the boundary condition (\ref{bc}) and the symmetry condition (\ref{sc}), the vertical component $w$ of the velocity can be uniquely determined as
\begin{equation}\label{*}
w(x,y,z,t)=-\int_0^z\nabla_H\cdot v(x,y,z',t)dz'.
\end{equation}

Set $u_0=(v_0, w_0)$, with $w_0$ given by (\ref{ne00}).
Then, it is obviously that $u_0\in L^2(\Omega)$
and $\nabla\cdot u_0=0$. Therefore, following the same arguments as those for the
standard Navier-Stokes equations, see, e.g., \cite{TEMAMNS,CONFO},
one can prove that there is a global weak solution, denoted by $u_\varepsilon=(v_\varepsilon, w_\varepsilon)$, to the scaled Navier-Stokes equations (\ref{SNS}), subject to the boundary and initial conditions (\ref{bc})--(\ref{sc}).

We are going to estimate the difference between $(v_\varepsilon, w_\varepsilon)$ and $(v,w)$. As a preparation, we need the following proposition, which, as it will be shown in the proof, is essentially obtained by testing the (SNS) against $(v,w)$.

\begin{proposition}\label{prop0}
Let $(v_\varepsilon, w_\varepsilon)$ and $(v,w)$ be the solutions of (SNS) and (PEs), with initial data $(v_0, w_0)$, $v_0\in H^1(\Omega)$ satisfying (\ref{E1}) and
$$
w_0(x,y,z)=-\int_0^z\nabla_H\cdot v_0(x,y,z') dz'.
$$
Then, the following integral equality holds
\begin{align}
  &-\frac{\varepsilon^2}{2}\|w(t)\|_2^2+\left(\int_\Omega v_\varepsilon\cdot v+\varepsilon^2w_\varepsilon w) dxdydz\right)(t)\nonumber\\
  &+\int_{Q_{t}}(-v_\varepsilon\cdot\partial_tv+\nabla v_\varepsilon:\nabla v +\varepsilon^2\nabla w_\varepsilon\cdot\nabla w) dxdydzds\nonumber\\
  =&\frac{\varepsilon^2}{2}\|w_0\|_2^2+\|v_0\|_2^2+
  \varepsilon^2\int_{Q_{t}} \left(\int_0^z \partial_tvdz'\right) \cdot\nabla_HW_\varepsilon dxdydzds\nonumber\\
  &-\int_{Q_{t}} [(u_\varepsilon\cdot\nabla)v_\varepsilon\cdot v+\varepsilon^2 u_\varepsilon\cdot\nabla w_\varepsilon w] dxdydzds,
  \label{ne4}
\end{align}
for any $t\in[0,\infty)$, where $Q_t=\Omega\times(0,t)$.
\end{proposition}

\begin{proof}
We will follow the argument of Serrin \cite{SERRIN} (see also \cite{BLNNT} and the references therein). Recalling the definition of weak solutions to (SNS), the following integral identity holds
\begin{align*}
\int_Q&[-(v_\varepsilon\cdot\partial_t\varphi_H +\varepsilon^2w_\varepsilon\partial_t\varphi_3)+((u_\varepsilon\cdot\nabla) v_\varepsilon \cdot\varphi_H+\varepsilon^2u_\varepsilon\cdot\nabla w_\varepsilon\varphi_3)+\nabla v_\varepsilon:\nabla\varphi_H\\
&+ \varepsilon^2\nabla w_\varepsilon\cdot\nabla\varphi_3]dxdydzdt= \int_\Omega(v_0\cdot\varphi_H(\cdot,0) +\varepsilon^2w_0\varphi_3(\cdot,0))dxdydz,
\end{align*}
for any periodic function $\varphi=(\varphi_H,\varphi_3)$, with $\varphi_H=(\varphi_1,\varphi_2)$, such that $\nabla\cdot\varphi=0$ and $\varphi\in C_0^\infty(\overline\Omega\times[0,\infty))$, where $Q:=\Omega\times(0,\infty)$.

Let $\chi\in C_0^\infty([0,\infty))$, with $0\leq\chi\leq1$ and $\chi(0)=1$, and set $\varphi=(v,w)\chi(t)$. We remark that, by the density argument, we can choose $\varphi$ as the testing function in the above integral identity, with modifying the term
$\int_Qw\partial_t\varphi_3dxdydzdt$ as
\begin{align*}
\int_Qw_\varepsilon \partial_t(w\chi)dxdydzdt=&\int_0^\infty\langle\partial_t(w\chi), w_\varepsilon\rangle_{H^{-1}\times H^1}dt.
\end{align*}
This is valid, because, recalling the regularities of $v$, and using (\ref{*}), we only have the regularity that
$\partial_tw\in L^2_{\text{loc}}([0,\infty); H^{-1}(\Omega))$. The validity
of the integrals involving the terms $v_\varepsilon\cdot\partial_t(v\chi), \nabla v_\varepsilon:\nabla(v\chi), \nabla w_\varepsilon\cdot\nabla (w\chi)$ is obviously guaranteed by the regularities of $u_\varepsilon$
and $(v,w)$, stated in the definition of the weak solutions and
(\ref{ne0}), respectively. The validity of the integral of the term
$(u_\varepsilon\cdot\nabla)v_\varepsilon\cdot v\chi$ follows from
utilizing the H\"older inequality and noticing that
$u_\varepsilon\in L^{\frac{10}{3}}_{\text{loc}}(\overline\Omega\times[0,\infty))$
and $v\in L^\infty_{\text{loc}}([0,\infty);L^6(\Omega))$, which
are easily verified by the interpolation and the embedding inequalities.
While the validity of the integral of the term $u_\varepsilon\cdot\nabla w_\varepsilon w\chi$ follows from the following calculation: denoting by $[0,T]$ the support set of $\chi$, and recalling (\ref{*}), it follows from Lemma \ref{LADTYPE} and the H\"older inequality that
\begin{align*}
  \int_Q|u_\varepsilon||\nabla w_\varepsilon| |w|\chi
  dxdydz\leq& \int_0^T\int_\Omega|u_\varepsilon||\nabla w_\varepsilon| \left|\int_0^z\nabla_H\cdot vdz'\right|dxdydzdt\\
  \leq&\int_0^T\int_M\int_{-1}^1|u_\varepsilon||\nabla w_\varepsilon|dz\int_{-1}^1|\nabla_Hv|dzdxdydt\\
  \leq&C\int_0^T\|u_\varepsilon\|_2^{\frac12}\|\nabla u_\varepsilon\|_2^{\frac12}\|\nabla w_\varepsilon\|_2\|\nabla_Hv\|_2^{\frac12}\|\Delta v\|_2^{\frac12}dt\\
  \leq&C\left(\sup_{0\leq t\leq T}\|u_\varepsilon\|_2\|\nabla v\|_2\right)^{\frac12}\left(\int_0^T\|\nabla u_\varepsilon\|_2^2dt\right)^{\frac12}\\
  &\times\left(\int_0^T\|\nabla w_\varepsilon\|_2^2dt\right)^{\frac12}
  \left(\int_0^T\|\Delta v\|_2^2dt\right)^{\frac12},
\end{align*}
where the Poincar\"e inequality has been used.

Combining the statements in the above paragraph, by taking $\varphi=(v,w)\chi$ as a testing function, we get the following integral identity
\begin{align*}
  \int_Q&[(-v_\varepsilon\cdot\partial_tv+\nabla v_\varepsilon:\nabla v +\varepsilon^2\nabla w_\varepsilon\cdot\nabla w)\chi-v_\varepsilon\cdot v\chi']dxdydzdt\\
  &-\varepsilon^2\int_0^\infty\langle\partial_t(w\chi), w_\varepsilon\rangle_{H^{-1}\times H^1}dt\\
  =&-\int_Q[(u_\varepsilon\cdot\nabla)v_\varepsilon\cdot v+\varepsilon^2 u_\varepsilon\cdot\nabla w_\varepsilon]\chi dxdydzdt+\|v_0\|_2^2+\varepsilon^2\|w_0\|_2^2.
\end{align*}
Let's rewrite the term $\int_0^\infty\langle\partial_t(w\chi), w_\varepsilon\rangle_{H^{-1}\times H^1}dt$ as
\begin{align*}
  \int_0^\infty\langle\partial_t(w\chi), w_\varepsilon\rangle_{H^{-1}\times H^1}dt
  =\int_0^\infty\langle\partial_tw,w_\varepsilon\rangle_{H^{-1}\times H^1}\chi dt+\int_Qw  w_\varepsilon\chi'dxdydzdt,
\end{align*}
which, substituted in the previous identity, gives
\begin{align}
  \int_Q&(-v_\varepsilon\cdot\partial_tv+\nabla v_\varepsilon:\nabla v +\varepsilon^2\nabla w_\varepsilon\cdot\nabla w)\chi dxdydzdt\nonumber\\
  &-\varepsilon^2\int_0^\infty\langle\partial_tw,w_\varepsilon \rangle_{H^{-1}\times H^1}\chi dt-\int_Q(v_\varepsilon\cdot v+\varepsilon^2w_\varepsilon w)\chi'dxdydzdt\nonumber\\
  =&-\int_Q[(u_\varepsilon\cdot\nabla)v_\varepsilon\cdot v+\varepsilon^2 u_\varepsilon\cdot\nabla w_\varepsilon w]\chi dxdydzdt+\|v_0\|_2^2+\varepsilon^2\|w_0\|_2^2,\label{ne1}
\end{align}
for any $\chi\in C_0^\infty([0,\infty))$, with $0\leq\chi\leq1$ and $\chi(0)=1$.

Given $t_0\in(0,\infty)$, and take a sufficient small positive number $\delta\in(0,t_0)$. Choose $\chi_\delta\in C_0^\infty([0,t_0)$, such that $\chi_\delta\equiv1$ on $[0,t_0-\delta]$, $0\leq\chi_\delta\leq1$ on $[t_0-\delta, t_0)$, and $|\chi_\delta'|\leq\frac2\delta$ on $[0,t_0)$. We claim that, as $\delta\rightarrow0$, we have
\begin{eqnarray}
&&\int_Q(v_\varepsilon\cdot v+\varepsilon^2w_\varepsilon w)\chi_\delta'dxdydzdt\rightarrow-\left(\int_\Omega (v_\varepsilon\cdot v+\varepsilon^2w_\varepsilon w)dxdydz\right)(t_0),\label{ne2}\\
&&\int_0^\infty\langle\partial_tw,w_\varepsilon \rangle_{H^{-1}\times H^1}\chi_\delta dt\rightarrow\int_0^{t_0}\langle\partial_tw,w_\varepsilon \rangle_{H^{-1}\times H^1}dt.\label{ne3}
\end{eqnarray}
The validity of (\ref{ne3}) follows from the dominant convergence theorem for the integrals, thanks to the observation
$$
\langle\partial_tw,w_\varepsilon\rangle=-\left\langle\nabla_H\cdot \left(\int_0^z\partial_tvdz'\right),w_\varepsilon\right\rangle =\int_\Omega \left(\int_0^z\partial_tvdz'\right)\cdot\nabla_Hw_\varepsilon dxdydz,
$$
which implies $\langle\partial_tw,w_\varepsilon\rangle\in L^1((0,t_0))$, here, for simplicity, we have dropped the subscript $H^{-1}\times H^1$.
While for (\ref{ne2}), by defining
$$
f(t):=\left(\int_\Omega(v_\varepsilon\cdot v+\varepsilon^2w_\varepsilon w)dxdydz\right)(t)
$$
it is equivalent to show $\int_{t_0-\delta}^{t_0}f(t)\chi_\delta'(t)dt\rightarrow -f(t_0)$.
Recalling the regularities that $u_\varepsilon\in C_w([0,\infty); L^2(\Omega))$ and
$v\in C([0,\infty); H^1(\Omega))$, hence one has $w\in C([0,\infty); L^2(\Omega))$, and thus $f$
is a continuous function on $[0,\infty)$. For any $\sigma>0$, by the continuity of $f$, there is a positive number $\rho$, such that $|f(t)-f(t_0)|\leq\sigma$, for any $t\in[t_0-\rho,t_0]$. Now, for any $\delta\in(0,\rho)$, recalling that $\chi_\delta\equiv1$ on $[0,t_0-\delta]$, $\chi_\delta(t_0)=0$, and $|\chi_\delta'|\leq\frac2\delta$ on $[0,\infty)$, we deduce
\begin{align*}
\left|\int_{t_0-\delta}^{t_0}f(t)\chi_\delta'(t)dt+f(t_0)\right| =&\left|\int_{t_0-\delta}^{t_0}(f(t)-f(t_0))\chi_\delta'(t)dt\right|\\
\leq&\int_{t_0-\delta}^{t_0}|f(t)-f(t_0)||\chi_\delta'(t)|dt
\leq2\sigma,
\end{align*}
which proves (\ref{ne2}).

Recalling $w=-\int_0^z\nabla_H\cdot vdz'$, and noticing that $w\in L^2_{\text{loc}}([0,\infty); H^1(\Omega))$ and $\partial_tw\in L^2_{\text{loc}}([0,\infty); H^{-1}(\Omega))$, we deduce
\begin{align*}
  \langle\partial_tw,w_\varepsilon\rangle =&\langle\partial_tw,w_\varepsilon-w\rangle +\langle\partial_tw,w\rangle\\
  =&\left\langle-\nabla_H\cdot\left(\int_0^z \partial_tvdz'\right),w_\varepsilon-w\right\rangle +\langle\partial_tw,w\rangle\\
  =&\int_\Omega\left(\int_0^z \partial_tvdz'\right)\cdot\nabla_HW_\varepsilon dxdydz +\frac12\frac{d}{dt}\|w\|_2^2,
\end{align*}
where the Lions--Magenes Lemma (see, e.g., pages 260--261 of \cite{TEMAMNS}) has been used, and thus
$$
\int_0^{t_0}\langle\partial_tw,w_\varepsilon\rangle dt
=\int_{Q_{t_0}} \left(\int_0^z \partial_tvdz'\right) \cdot\nabla_HW_\varepsilon  dxdydzdt
+\frac{1}{2}(\|w(t_0)\|_2^2-\|w_0\|_2^2),
$$
where $Q_{t_0}=\Omega\times(0,t_0)$.
Thanks to the above equality and (\ref{ne2})--(\ref{ne3}), one can choose $\chi=\chi_\delta$ in (\ref{ne1}), as in the previous paragraph, and let $\delta$ goes to zero to get
\begin{align*}
  &-\frac{\varepsilon^2}{2}\|w(t_0)\|_2^2+\left(\int_\Omega v_\varepsilon\cdot v+\varepsilon^2w_\varepsilon w) dxdydz\right)(t_0)\nonumber\\
  &+\int_{Q_{t_0}}(-v_\varepsilon\cdot\partial_tv+\nabla v_\varepsilon:\nabla v +\varepsilon^2\nabla w_\varepsilon\cdot\nabla w) dxdydzdt\nonumber\\
  =&\frac{\varepsilon^2}{2}\|w_0\|_2^2+\|v_0\|_2^2+
  \varepsilon^2\int_{Q_{t_0}} \left(\int_0^z \partial_tvdz'\right) \cdot\nabla_HW_\varepsilon dxdydzdt\nonumber\\
  &-\int_{Q_{t_0}} [(u_\varepsilon\cdot\nabla)v_\varepsilon\cdot v+\varepsilon^2 u_\varepsilon\cdot\nabla w_\varepsilon w] dxdydzdt,
\end{align*}
for any $t_0\in[0,\infty)$. This completes the proof.
\end{proof}

Now, we can estimate the difference between $(v_\varepsilon, w_\varepsilon)$ and $(v,w)$.

\begin{proposition}
  \label{prop}
Under the same assumptions as in Proposition \ref{prop0} and denoting $(V_\varepsilon, W_\varepsilon):=(v_\varepsilon-v, w_\varepsilon-w)$, the following holds
  \begin{align*}
  \sup_{0\leq t<\infty}(\|V_\varepsilon\|_2^2&+\varepsilon^2\|W_\varepsilon\|_2^2)(t) +\int_0^{\infty} (\|\nabla V_\varepsilon\|_2^2+\varepsilon^2\|\nabla W_\varepsilon \|_2^2)ds\nonumber\\
  \leq& C(\|v_0\|_{H^1}, L_1,L_2)\varepsilon^2(\|v_0\|_2^2+\varepsilon^2\|w_0\|_2^2+1)^2,
  \end{align*}
  where $C(\|v_0\|_{H^1}, L_1,L_2)$ denotes a constant depending only on $\|v_0\|_{H^1}$, $L_1$ and $L_2$.
\end{proposition}

\begin{proof}
Multiplying equation (\ref{PE}) by $v_\varepsilon$, integrating the resultant over $Q_{t_0}$, it follows from integration by parts that
\begin{equation}
  \label{ne5}
  \int_{Q_{t_0}}(\partial_tv\cdot v_\varepsilon+\nabla v:\nabla v_\varepsilon)dxdydzdt=-\int_{Q_{t_0}}(u\cdot\nabla)v\cdot v_\varepsilon dxdydzdt,
\end{equation}
for any $t_0\in[0,\infty)$. Multiplying equation (\ref{PE}) by $v$, integrating the resultant over $Q_{t_0}$, the it follows from integration by parts that
\begin{equation}
  \label{ne6}
  \frac12\|v(t_0)\|_2^2+\int_0^{t_0}\|\nabla v\|_2^2dt=\frac12\|v_0\|_2^2,
\end{equation}
for any $t_0\in[0,\infty)$.
By the definition of Leray-Hopf weak solutions, one has
\begin{align}
  \frac12(\|v_\varepsilon(t_0)\|_2^2 +\varepsilon^2\|w_\varepsilon(t_0)\|_2^2)&+\int_0^{t_0} (\|\nabla v_\varepsilon\|_2^2 +\varepsilon^2\|\nabla w_\varepsilon\|_2^2)ds\nonumber\\
  \leq&\frac12(\|v_0\|_2^2+\varepsilon^2\|w_0\|_2^2),\label{ne7}
\end{align}
for a.e. $t_0\in[0,\infty)$.

Summing (\ref{ne6}) and (\ref{ne7}), then subtracting from the resultant (\ref{ne4}) (choose $t=t_0$ there) and (\ref{ne5}) yields
\begin{align}
&\frac12(\|V_\varepsilon\|_2^2+\varepsilon^2\|W_\varepsilon\|_2^2)(t_0)+\int_0^{t_0} (\|\nabla V_\varepsilon\|_2^2+\varepsilon^2\|\nabla W_\varepsilon \|_2^2)dt\nonumber\\
  \leq&-\varepsilon^2\int_{Q_{t_0}} \left[\left(\int_0^z \partial_tvdz'\right) \cdot\nabla_HW_\varepsilon +\nabla w\cdot\nabla W_\varepsilon \right]dxdydzdt\nonumber\\
  &+\int_{Q_{t_0}} [(u_\varepsilon\cdot\nabla)v_\varepsilon\cdot v+(u\cdot\nabla)v\cdot v_\varepsilon]dxdydzdt\nonumber\\
  &+\varepsilon^2\int_{Q_{t_0}} u_\varepsilon\cdot\nabla w_\varepsilon w ~dxdydzdt=:I_1+I_2+I_3,\label{ne8}
\end{align}
for a.e. $t_0\in[0,\infty)$.
It follows from the H\"older and Cauchy-Schwarz inequalities, and using Corollary \ref{cor} that
\begin{align}
I_1\leq& \varepsilon^2(\|\partial_tv\|_{L^2(Q_{t_0})}+\|\nabla w\|_{L^2(Q_{t_0})})\|\nabla W_\varepsilon\|_{L^2(Q_{t_0})}\nonumber\\
\leq&\frac{\varepsilon^2 }{6}\|\nabla W_\varepsilon\|_{L^2(Q_{t_0})}^2 +C(\|v_0\|_{H^1},L_1,L_2)\varepsilon^2.\label{I1}
\end{align}

We are going to estimate the quantities $I_2$ and $I_3$ on the right-hand side of (\ref{ne8}).
Using the incompressibility conditions, it follows from integration by parts that
\begin{align*}
  I_2:=&\int_{Q_{t_0}}[(u_\varepsilon\cdot\nabla)v_\varepsilon\cdot v+(u\cdot\nabla)v\cdot v_\varepsilon ]dxdydzdt\\
  =&\int_{Q_{t_0}}[(u_\varepsilon\cdot\nabla)v_\varepsilon\cdot v-(u\cdot\nabla)v_\varepsilon \cdot v]dxdydzdt\\
  =&\int_{Q_{t_0}}[(u_\varepsilon-u)\cdot\nabla]v_\varepsilon\cdot vdxdydzdt\\
  =&\int_{Q_{t_0}}[(u_\varepsilon-u)\cdot\nabla]V_\varepsilon\cdot vdxdydzdt.
\end{align*}
The quantity $I_2$ will be divided into two parts $I_2'$ and $I_2''$, below. It follows from the H\"older, Sobolev and Young inequalities that
\begin{align*}
  I_{2}':=&\int_{Q_{t_0}}(V_\varepsilon\cdot\nabla_H)V_\varepsilon\cdot vdxdydzdt\\
  \leq&\int_0^{t_0}\|V_\varepsilon\|_3\|\nabla V_\varepsilon\|_2 \|v\|_6dt
  \leq C\int_0^{t_0}\|V_\varepsilon\|_2^{\frac12} \|\nabla V_\varepsilon\|_2^{\frac32}\|\nabla v\|_2dt\\
  \leq&\frac{1}{18}\|\nabla V_\varepsilon\|_{L^2(Q_{t_0})}^2 +C\int_0^{t_0}\|\nabla v\|_2^4\|V_\varepsilon\|_2^2dt.
\end{align*}
Integration by parts yields
\begin{align*}
  I_2'':=&\int_{Q_{t_0}}W_\varepsilon\partial_zV_\varepsilon\cdot vdxdydzdt\\
  =&-\int_{Q_{t_0}}[\partial_zW_\varepsilon V_\varepsilon\cdot v +W_\varepsilon V_\varepsilon\cdot\partial_zv]dxdydzdt\\
  =&\int_{Q_{t_0}}[\nabla_H\cdot V_\varepsilon V_\varepsilon\cdot v-W_\varepsilon V_\varepsilon\cdot\partial_zv]dxdydzdt\\
\end{align*}
For the first term of $I_{2}''$, denoted by $I_{21}''$, the same arguments as for $I_2'$ yield
\begin{align*}
  I_{21}'':=&\int_{Q_{t_0}}(\nabla_H\cdot V_\varepsilon) (V_\varepsilon\cdot v)dxdydzdt\\
  \leq&\frac{1}{18}\|\nabla V_\varepsilon\|_{L^2(Q_{t_0})}^2 +C\int_0^{t_0}\|\nabla v\|_2^4\|V_\varepsilon\|_2^2dt.
\end{align*}
For the second term of $I_{2}''$, denoted by $I_{22}''$, by Lemma \ref{LADTYPE}, it follows from the Poincar\'e and Young inequalities that
\begin{align*}
  I_{22}'':=&-\int_{Q_{t_0}}W_\varepsilon V_\varepsilon\cdot\partial_zv dxdydzdt\\
  =&\int_0^t\int_\Omega\left(\int_0^z\nabla_H\cdot V_\varepsilon dz'\right)( V_\varepsilon\cdot\partial_zv) dxdydzdt\\
  \leq&\int_0^{t_0}\int_M\left(\int_{-1}^1|\nabla_H V_\varepsilon|dz\right)\left(\int_{-1}^1| V_\varepsilon||\partial_zv|dz\right) dxdydt\\
  \leq&C\int_0^{t_0}\|\nabla V_\varepsilon \|_2^{\frac32}\|V_\varepsilon\|_2^{\frac12}\|\nabla v\|_2^{\frac12}
  \|\Delta v\|_2^{\frac12}dt\\
  \leq&\frac{1}{18}\|\nabla V_\varepsilon\|_{L^2(Q_{t_0})}^2 +C\int_0^{t_0}\|\nabla v\|_2^2\|\Delta v\|_2^2\|V_\varepsilon\|_2^2dt.
\end{align*}
Thanks to the estimates for $I_2'$, $I_{21}''$ and $I_{22}''$, we can bound $I_2$ as
\begin{equation}
  I_2 \leq \frac{1}{6}\|\nabla V_\varepsilon \|_{L^2(Q_{t_0})}^2 +C\int_0^{t_0}\|\nabla v\|_2^2\|\Delta v\|_2^2\|V_\varepsilon\|_2^2dt,\label{I2}
\end{equation}
note that the Poincar\'e inequality has been used.

We still need to estimate the last term $I_3$ in (\ref{ne8}). Using the incompressibility conditions, we deduce that
\begin{align*}
  I_3:=&\varepsilon^2\int_{Q_{t_0}} u_\varepsilon\cdot\nabla w_\varepsilon w dxdydzdt
  =\varepsilon^2\int_{Q_{t_0}} u_\varepsilon\cdot\nabla W_\varepsilon w dxdydzdt\\
  =&\varepsilon^2\int_{Q_{t_0}} [v_\varepsilon\cdot\nabla_H W_\varepsilon-w_\varepsilon\nabla_H\cdot V_\varepsilon]w dxdydzdt\\
  \leq&\varepsilon^2\int_0^{t_0}\int_M\left(\int_{-1}^1 (|v_\varepsilon||\nabla_H W_\varepsilon~|+|w_\varepsilon||\nabla_H V_\varepsilon|)dz \right)\left(\int_{-1}^1|\nabla_Hv|dz\right)dxdydt.
\end{align*}
Thus, it follows from Lemma \ref{LADTYPE}, the Poincar\'e and Young inequalities that
\begin{align*}
  I_3\leq&C\varepsilon^2\int_0^{t_0}(\|v_\varepsilon\|_2^{\frac12}\|\nabla v_\varepsilon\|_2^{\frac12}\|\nabla W_\varepsilon\|_2\\
  &+ \|w_\varepsilon\|_2^{\frac12}\|\nabla w_\varepsilon\|_2^{\frac12}
  \|\nabla_HV_\varepsilon\|_2)\|\nabla v\|_2^{\frac12}\|\Delta v\|_2^{\frac12}dt\\
  \leq&C\varepsilon^2\int_0^{t_0} (\|v_\varepsilon\|_2^2\|\nabla v_\varepsilon\|_2^2+\|\nabla v\|_2^2\|\Delta v\|_2^2 +\varepsilon^2 \|w_\varepsilon\|_2^2\|\nabla w_\varepsilon\|_2^2)dt\\
  &+\frac16\left(\|\nabla V_\varepsilon\|_{L^2(Q_{t_0})}^2 +\varepsilon^2\|\nabla W_ \varepsilon\|_{L^2(Q_{t_0})}^2 \right),
\end{align*}
from which, recalling (\ref{ne7}) and by Corollary \ref{cor}, we have
\begin{align}
  I_3\leq& \frac16\left(\|\nabla V_\varepsilon\|_{L^2(Q_{t_0})}^2 +\varepsilon^2\|\nabla W_\varepsilon\|_{L^2(Q_{t_0})}^2 \right) \nonumber\\
  &+C\varepsilon^2[(\|v_0\|_2^2+\varepsilon^2\|w_0\|_2^2)^2+C(\|v_0\|_{H^1}, L_1,L_2)]\label{I3}
\end{align}

Substituting (\ref{I1})--(\ref{I3}) into (\ref{ne8}) yields
\begin{align*}
f(t):=(\|V_\varepsilon\|_2^2&+\varepsilon^2\|W_\varepsilon\|_2^2)(t)+\int_0^{t} (\|\nabla V_\varepsilon\|_2^2+\varepsilon^2\|\nabla W_\varepsilon \|_2^2)ds\nonumber\\
  \leq&C\varepsilon^2[(\|v_0\|_2^2+\varepsilon^2\|w_0\|_2^2)^2 +C(\|v_0\|_{H^1}, L_1,L_2)]\nonumber\\
  &+C\int_0^{t}\|\nabla v\|_2^2\|\Delta v\|_2^2\|V_\varepsilon\|_2^2ds=:F(t),
\end{align*}
for a.e. $t\in[0,\infty)$. Thus, we have
\begin{align*}
F'(t)=&C\|\nabla v\|_2^2\|\Delta v\|_2^2\|V_\varepsilon\|_2^2\\
\leq& C\|\nabla v\|_2^2\|\Delta v\|_2^2f(t)\leq C\|\nabla v\|_2^2\|\Delta v\|_2^2F(t),
\end{align*}
from which, by the Gronwall inequality, and using Corollary \ref{cor}, we deduce
\begin{align*}
  f(t)\leq&F(t)\leq e^{C\int_0^t\|\nabla v\|_2^2\|\Delta v\|_2^2ds}F(0)\\
  \leq&\varepsilon^2 C(\|v_0\|_{H^1}, L_1,L_2)(\|v_0\|_2^2+\varepsilon^2\|w_0\|_2^2+1)^2,
\end{align*}
which implies the conclusion.
\end{proof}

With the aid of Proposition \ref{prop}, we can now give the proof of Theorem \ref{thm0}.

\begin{proof}[\textbf{Proof of Theorem \ref{thm0}}]
The estimate in Theorem \ref{thm0} follows from Proposition \ref{prop}, while the convergence is a direct consequence of that estimate.
\end{proof}

\section{Strong convergence II: the $H^2$ initial data case}
\label{sec5}
In this section, we prove the strong convergence of (SNS) to (PEs), with initial data $v_0\in H^2(\Omega)$, as the aspect ration parameter $\varepsilon$ goes to zero. In other words, we give the proof of Theorem \ref{thm}.

Let $v_0\in H^2(\Omega)$, and suppose that
$$
\nabla_H\cdot\left(\int_{-1}^1 v_0(x,y,z)dz\right)=0,\quad\mbox{for all }(x,y)\in M.
$$
Set $u_0=(v_0,w_0)$, with $w_0$ given by (\ref{ne00}), then $u_0\in H^1(\Omega)$ and $\nabla\cdot u_0=0$.
By the same arguments as those for the standard Navier-Stokes equations, see, e.g., Constantin--Foias \cite{CONFO} and Temam \cite{TEMAMNS}, one can prove that, there is a unique local (in time) strong solution $u_\varepsilon=(v_\varepsilon, w_\varepsilon)$ to (SNS), subject to (\ref{bc})--(\ref{sc}). Denote by $T_\varepsilon^*$ the maximal existence time of the strong solution $(v_\varepsilon, w_\varepsilon)$.
Let $u=(v,w)$ be the unique solutions to (PEs), subject to (\ref{bc})--(\ref{sc}).

Denote, as before, $U_\varepsilon$ the difference between
$u_\varepsilon$ and $u$, that is
$$
U_\varepsilon=(V_\varepsilon,W_\varepsilon),\quad V_\varepsilon=v_\varepsilon-v,\quad W_\varepsilon=w_\varepsilon-w.
$$
Then, one can easily verify that $U_\varepsilon=(V_\varepsilon,W_\varepsilon)$ satisfies the following system
\begin{eqnarray}
  &\partial_tV_\varepsilon+(U_\varepsilon\cdot\nabla)V_\varepsilon-\Delta V_\varepsilon+\nabla_HP_\varepsilon+( u\cdot\nabla)V_\varepsilon+(U_\varepsilon\cdot\nabla) v=0,\label{DIFF-1}\\
  &\nabla_H\cdot V_\varepsilon+\partial_zW_\varepsilon=0,\label{DIFF-2}\\
  &\varepsilon^2(\partial_tW_\varepsilon+U_\varepsilon\cdot\nabla W_\varepsilon-\Delta W_\varepsilon+U_\varepsilon\cdot\nabla w+  u\cdot\nabla W_\varepsilon)+\partial_zP_\varepsilon\nonumber\\
  &=-\varepsilon^2(\partial_t w+  u\cdot\nabla  w-\Delta  w),\label{DIFF-3}
\end{eqnarray}
in $\Omega\times(0,T_\varepsilon^*)$. Due to the smoothing effect of (SNS) to the unique strong solutions,
one can show that the strong solution $(v_\varepsilon, w_\varepsilon)$ is smooth in the time interval $(0, T_\varepsilon^*)$, and thus, recalling that $(v,w)$ is smooth away from the initial time, so is $(V_\varepsilon, W_\varepsilon)$. This guarantees the validity of the arguments in the proof below.

We are going to do the a priori estimates on $(V_\varepsilon, W_\varepsilon)$. We start with the basic energy estimate stated in the following proposition.

\begin{proposition}[Basic $L^2$ energy estimate]
  \label{basicdiff}
  The following basic energy estimate holds
  \begin{align*}
  \sup_{0\leq s\leq t}(\|V_\varepsilon\|_2^2&+\varepsilon^2\|W_\varepsilon\|_2^2) +\int_0^t (\|\nabla V_\varepsilon\|_2^2+\varepsilon^2\|\nabla W_\varepsilon \|_2^2)ds\nonumber\\
  \leq& C\varepsilon^2(\|v_0\|_2^2+\varepsilon^2\|w_0\|_2^2+1)^2,
  \end{align*}
for any $t\in[0,T_\varepsilon^*)$, where $C$ is a constant depending only on $\|v_0\|_{H^1}$, $L_1$ and $L_2$.
\end{proposition}

\begin{proof}
This is a direction consequence of Proposition \ref{prop}.
\end{proof}

The first order energy estimate is stated in the following proposition.

\begin{proposition}[$H^1$ energy estimates]
  \label{firstdiff}
There exists a positive constant $\delta_0$ depending only on $L_1$ and $L_2$, such that, the following estimate holds
  \begin{align*}
    &\sup_{0\leq s\leq t}(\|\nabla V_\varepsilon\|_2^2+\varepsilon^2\|\nabla W_\varepsilon\|_2^2)+\int_0^t(\|\Delta V_\varepsilon\|_2^2+\varepsilon^2\|\Delta W_\varepsilon \|_2^2) ds\\
    \leq&C\varepsilon^2e^{C(1+\varepsilon^4)\int_0^t\|\Delta  v\|_2^2\|\nabla\Delta  v\|_2^2ds}\int_0^t(1+\|\Delta  v\|_2^2) (\|\nabla\partial_t  v\|_2^2+\|\nabla\Delta  v\|_2^2)ds,
  \end{align*}
  for any $t\in[0,T_\varepsilon^*)$, as long as
  $$
  \sup_{0\leq s\leq t}(\|\nabla V_\varepsilon\|_2^2+\varepsilon^2\|\nabla W_\varepsilon\|_2^2)\leq\delta_0^2,
  $$
  where $C$ is a positive constant depending only on $L_1$ and $L_2$.
\end{proposition}

\begin{proof}
For simplicity of the notations, we drop the subscript index
$\varepsilon$ of $(V_\varepsilon, W_\varepsilon)$ in the following
proof, in other words, we use $(V,W)$ to replace $(V_\varepsilon, W_\varepsilon)$.

Taking the $L^2(\Omega)$ inner products to equations (\ref{DIFF-1}) and (\ref{DIFF-3}) with $-\Delta V$ and $-\Delta W$, respectively, summing the resultants up and integration by parts yield
  \begin{align}
    &\frac12\frac{d}{dt}(\|\nabla V\|_2^2+\|\varepsilon\nabla W\|_2^2)+\|\Delta V\|_2^2+\|\varepsilon\Delta W\|_2^2\nonumber\\
    =&\int_\Omega[(U\cdot\nabla)V+(  u\cdot\nabla)V+(U\cdot\nabla)  v]\cdot\Delta Vdxdydz\nonumber\\
    &+\varepsilon^2\int_\Omega(U\cdot\nabla W+  u\cdot\nabla W+U\cdot\nabla  w)\Delta Wdxdydz\nonumber\\
    &+\varepsilon^2\int_\Omega(\partial_t  w+  u\cdot\nabla  w-\Delta  w)\Delta Wdxdydz. \label{bd2}
  \end{align}

  We are going to estimate the terms on the right-hand side of (\ref{bd2}).
  First, by Lemma \ref{ladlem}, it follows from the Young and Poincar\'e inequalities that
  \begin{align}
    &\int_\Omega[(U\cdot\nabla)V+(  u\cdot\nabla)V+(U\cdot\nabla)  v]\cdot\Delta Vdxdydz\nonumber\\
    \leq&C(\|\nabla V\|_2\|\Delta V\|_2+\|\nabla  v\|_2^{\frac12}\|\Delta  v\|_2^{\frac12}\|\nabla V\|_2^{\frac12}\|\Delta V\|_2^{\frac12})\|\Delta V\|_2\nonumber\\
    \leq&\frac{1}{10}\|\Delta V\|_2^2+C(\|\nabla V\|_2^2\|\Delta V\|_2^2+\|\nabla  v\|_2^2\|\Delta  v\|_2^2\|\nabla V\|_2^2)\nonumber\\
    \leq&\frac{1}{10}\|\Delta V\|_2^2+C(\|\nabla V\|_2^2\|\Delta V\|_2^2+\|\Delta  v\|_2^2\|\nabla\Delta  v\|_2^2\|\nabla V\|_2^2),\label{bd2-1}
  \end{align}
  and
  \begin{align}
    &\varepsilon^2\int_\Omega(U\cdot\nabla W+  u\cdot\nabla W+U\cdot\nabla  w)\Delta Wdxdydz\nonumber\\
    \leq&C\varepsilon^2[(\|\nabla V\|_2^{\frac12}\|\Delta V\|_2^{\frac12}+\|\nabla  v\|_2^{\frac12}\|\Delta  v\|_2^{\frac12})\|\nabla W\|_2^{\frac12}\|\Delta W\|_2^{\frac12}\nonumber\\
    &+\|\nabla V\|_2^{\frac12}\|\Delta V\|_2^{\frac12}\|\nabla  w\|_2^{\frac12}\|\Delta  w\|_2^{\frac12}]\|\Delta W\|_2\nonumber\\
    \leq&\frac{1}{10}(\|\Delta V\|_2^2+\|\varepsilon\Delta W\|_2^2)+C(\|\nabla V\|_2^2\|\Delta V\|_2^2+\|\varepsilon\nabla W\|_2^2\|\varepsilon\Delta W\|_2^2)\nonumber\\
    &+C\|\nabla  v\|_2^2\|\Delta  v\|_2^2\|\varepsilon\nabla W\|_2^2+C\varepsilon^4\|\nabla  w\|_2^2\|\Delta  w\|_2^2\|\nabla V\|_2^2\nonumber\\
    \leq&\frac{1}{10}(\|\Delta V\|_2^2+\|\varepsilon\Delta W\|_2^2)+C (\|\nabla V\|_2^2+\|\varepsilon\nabla W\|_2^2)\nonumber\\
    &\times[\|\Delta V\|_2^2+\|\varepsilon\Delta W\|_2^2+(1+\varepsilon^4)\|\Delta  v\|_2^2\|\nabla\Delta  v\|_2^2],\label{bd2-2}
  \end{align}
  where in the last step we have used the fact that
  $$
  \|\nabla  w\|_2\leq C\|\Delta  v\|_2,\quad\|\Delta  w\|_2\leq C\|\nabla\Delta  v\|_2,
  $$
  which can be easily verified by recalling $  w(x,y,z,t)=-\int_0^z\nabla_H\cdot  v(x,y,z',t)dz'$ and using the Poincar\'e inequality.
  Next, using again Lemma \ref{ladlem}, it follows from the H\"older, Young and Poincar\'e inequalities that
  \begin{align}
    &\varepsilon^2\int_\Omega(\partial_t  w+  u\cdot\nabla  w-\Delta  w)\Delta Wdxdydz\nonumber\\
    \leq&\varepsilon^2(\|\partial_t  w\|_2+\|\Delta  w\|_2)\|\Delta W\|_2+C\varepsilon^2\|\nabla  v\|_2^{\frac12}\|\Delta  v\|_2^{\frac12}\|\nabla  w\|_2^{\frac12}\|\Delta  w\|_2^{\frac12}\|\Delta W\|_2\nonumber\\
    \leq&\frac{1}{5}\|\varepsilon\Delta W\|_2^2+C\varepsilon^2(\|\partial_t  w\|_2^2+\|\Delta  w\|_2^2+\|\nabla  v\|_2^2\|\Delta  v\|_2^2+\|\nabla  w\|_2^2\|\Delta  w\|_2^2)\nonumber\\
    \leq&\frac{1}{5}\|\varepsilon\Delta W\|_2^2+C\varepsilon^2(\|\nabla\partial_t  v\|_2^2+\|\nabla\Delta  v\|_2^2+\|\Delta  v\|_2^2\|\nabla\Delta  v\|_2^2).\label{bd2-3}
  \end{align}

  Substituting the estiates (\ref{bd2-1})--(\ref{bd2-3}) into (\ref{bd2}) yields
  \begin{align*}
    &\frac12\frac{d}{dt}(\|\nabla V\|_2^2+\|\varepsilon\nabla W\|_2^2)+\frac35(\|\Delta V\|_2^2+\|\varepsilon\Delta W\|_2^2)\nonumber\\
    \leq&C_1 (\|\nabla V\|_2^2+\|\varepsilon\nabla W\|_2^2)
     [\|\Delta V\|_2^2+\|\varepsilon\Delta W\|_2^2+(1+\varepsilon^4)\|\Delta  v\|_2^2\|\nabla\Delta  v\|_2^2]\nonumber\\
     &+C_1\varepsilon^2(1+\|\Delta  v\|_2^2)(\|\nabla\partial_t  v\|_2^2+\|\nabla\Delta  v\|_2^2),
  \end{align*}
  for a positive constant $C_1$ depending only on $L_1$ and $L_2$.

  By the assumption $\sup_{0\leq s\leq t}(\|\nabla V\|_2^2+\|\varepsilon\nabla W\|_2^2)\leq\delta_0^2$. Choosing $\delta_0=\sqrt{\frac{1}{10C_1}}$, it follows from the above inequality that
  \begin{align*}
    &\frac{d}{dt}(\|\nabla V\|_2^2+\|\varepsilon\nabla W\|_2^2)+\|\Delta V\|_2^2+\|\varepsilon\Delta W\|_2^2\nonumber\\
    \leq&2C_1 (1+\varepsilon^4)
    \|\Delta  v\|_2^2\|\nabla\Delta  v\|_2^2(\|\nabla V\|_2^2+\|\varepsilon\nabla W\|_2^2)\nonumber\\
     &+2C_1\varepsilon^2(1+\|\Delta  v\|_2^2)(\|\nabla\partial_t  v\|_2^2+\|\nabla\Delta  v\|_2^2),
  \end{align*}
  from which, recalling $(V,W)|_{t=0}=0$, it follows from the Gronwall inequality that
  \begin{align*}
    &\sup_{0\leq s\leq t}(\|\nabla V\|_2^2+\varepsilon^2\|\nabla W\|_2^2)+\int_0^t(\|\Delta V\|_2^2+\varepsilon^2\|\Delta W\|_2^2) ds\\
    \leq&2C_1\varepsilon^2e^{2C_1(1+\varepsilon^4)\int_0^t\|\Delta  v\|_2^2\|\nabla\Delta  v\|_2^2ds}\int_0^t(1+\|\Delta  v\|_2^2) (\|\nabla\partial_t  v\|_2^2+\|\nabla\Delta  v\|_2^2)ds,
  \end{align*}
  proving the conclusion.
\end{proof}

Thanks to Propositions \ref{basicdiff}--\ref{firstdiff}, as well as Corollary \ref{cor}, we can prove the following:

\begin{proposition}
  \label{diffest}
There is a positive constant $\varepsilon_0$ depending only on $\|v_0\|_{H^2}, L_1$ and $L_2$, such that for any $\varepsilon\in(0,\varepsilon_0)$, there is a unique global strong solution $(v_\varepsilon, w_\varepsilon)$ to (SNS), subject to
(\ref{bc})--(\ref{sc}). Moreover, the following estimate holds
$$
\sup_{0\leq t<\infty}(\|V_\varepsilon\|_{H^1}^2+\varepsilon^2 \|W_\varepsilon\|_{H^1}^2) +\int_0^\infty(\|\nabla V_\varepsilon\|_{H^1}^2+\varepsilon^2\|\nabla W_\varepsilon\|_{H^1}^2)dt\leq C\varepsilon^2,
$$
where $C$ is a positive constant depending only on $\|v_0\|_{H^2}$, $L_1$ and $L_2$.
\end{proposition}

\begin{proof}
Recall that $T_\varepsilon^*$ is the maximal existence time of the strong solutions $(v_\varepsilon, w_\varepsilon)$ to (SNS), subject to the boundary and initial conditions (\ref{bc})--(\ref{sc}). By Corollary \ref{cor} and Proposition \ref{basicdiff}, we have the estimate
\begin{equation}
\sup_{0\leq t<T_\varepsilon^*}(\|V_\varepsilon \|_2^2+\varepsilon^2\|W_\varepsilon\|_2^2)+\int_0^{T_\varepsilon^*} (\|\nabla V_\varepsilon\|_2^2+\varepsilon^2\|\nabla W_\varepsilon\|_2^2)dt\leq K_1\varepsilon^2, \label{est1}
\end{equation}
where $K_1$ is a positive constant depending only on $\|v_0\|_{H^1}$, $L_1$ and $L_2$.

Let $\delta_0$ be the constant in Proposition \ref{firstdiff}, which depends only on $L_1$ and $L_2$. Define
$$
t_\varepsilon^*:=\sup\left\{t\in(0,T_\varepsilon^*)~\bigg|~\sup_{0\leq s\leq t}
(\|\nabla V_\varepsilon\|_2^2+\varepsilon^2\|\nabla W_\varepsilon\|_2^2)\leq \delta_0^2\right\}.
$$
By Proposition \ref{firstdiff} and Corollary \ref{cor}, we have the estimate
\begin{equation}\label{est2}
\sup_{0\leq s\leq t}(\|\nabla V_\varepsilon\|_2^2+\varepsilon^2\|\nabla W_\varepsilon\|_2^2)+\int_0^t(\|\Delta V_\varepsilon\|_2^2+\varepsilon^2\|\Delta W_\varepsilon\|_2^2)ds\leq K_2 \varepsilon^2,
\end{equation}
for any $t\in[0,t_\varepsilon^*)$, where $K_2$ is a positive constant depending only on $\|v_0\|_{H^2}$, $L_1$ and $L_2$. Setting $\varepsilon_0=\sqrt{\frac{\delta_0}{2K_2}}$, then the above inequality implies
$$
\sup_{0\leq s\leq t}(\|\nabla V_\varepsilon\|_2^2+\varepsilon^2\|\nabla W_\varepsilon\|_2^2)+\int_0^t(\|\Delta V_\varepsilon\|_2^2+\varepsilon^2\|\Delta W_\varepsilon\|_2^2)ds\leq \frac{\delta_0}{2},
$$
for any $\varepsilon\in(0,\varepsilon_0)$, and for any $t\in[0,t_\varepsilon^*)$, which, in particular, gives
$$
\sup_{0\leq t<t_\varepsilon^*}(\|\nabla V_\varepsilon\|_2^2+\varepsilon^2\|\nabla W_\varepsilon\|_2^2)\leq \frac{\delta_0}{2}.
$$
Thus, by the definition of $t_\varepsilon^*$, we must have
$t_\varepsilon^*=T_\varepsilon^*$. Thanks to this, it is clear that
(\ref{est2}) holds for any $t\in[0, T_\varepsilon^*)$.

We
claim that it must has $T_\varepsilon^*=\infty$, otherwise, if
$T_\varepsilon^*<\infty$, then, recalling that (\ref{est2}) holds for any $t\in[0, T_\varepsilon^*)$,
by the local well-posedness result of the
(SNS), one can extend the strong solution $(v_\varepsilon,
w_\varepsilon)$ beyond $T_\varepsilon^*$, which contradicts to the
definition of $T_\varepsilon^*$. Therefore, the conclusion follows by combining (\ref{est1}) with (\ref{est2}).
\end{proof}

Based on Proposition \ref{diffest}, we can now give the proof of Theorem \ref{thm} as follows:

\begin{proof}[\textbf{Proof of Theorem \ref{thm}}]
Let $\varepsilon_0$ be the constant in Proposition \ref{diffest}, which
depends only on $\|v_0\|_{H^2}$, $L_1$ and $L_2$. Then, by
Proposition \ref{diffest}, for any $\varepsilon\in(0,\varepsilon_0)$,
there is a unique global strong solution $(v_\varepsilon, w_\varepsilon)$ to (SNS), subject to the boundary and initial
conditions (\ref{bc})--(\ref{sc}). Moreover, the following
estimate holds
$$
\sup_{0\leq t<\infty}(\|V_\varepsilon\|_{H^1}^2+\varepsilon^2 \|W_\varepsilon\|_{H^1}^2) +\int_0^\infty(\|\nabla V_\varepsilon\|_{H^1}^2+\varepsilon^2\|\nabla W_\varepsilon\|_{H^1}^2)dt\leq C\varepsilon^2,
$$
where $(V_\varepsilon, W_\varepsilon)=(v_\varepsilon, w_\varepsilon)-(v,w)$, and $C$ is a positive constant depending
only on $\|v_0\|_{H^2}$, $L_1$ and $L_2$. This proves the estimates
stated in the theorem, while the strong convergencs stated there are
just the direct corollaries of this estimate. This completes the proof
of Theorem \ref{thm}.
\end{proof}

\section*{Acknowledgments}
{The work of J.L. is supported in part by the Direct Grant for Research 2016/2017 (Project Code: 4053216) from The Chinese University of Hong Kong. The work of Edriss S.T. is supported in part by the ONR grant N00014-15-1-2333.}
\par

\end{document}